\documentclass[reqno,12pt]{amsart}
\usepackage[colorlinks=true, linkcolor=blue, citecolor=blue]{hyperref}

\usepackage{amssymb}
\usepackage{amsmath, graphicx, rotating}
\usepackage{color}
\usepackage{soul}
\usepackage[dvipsnames]{xcolor}

\allowdisplaybreaks

\usepackage{ifthen}
\usepackage{xkeyval}
\usepackage{todonotes}
\usepackage[T1]{fontenc}
\usepackage{lmodern}
\usepackage[english]{babel}

\usepackage{ upgreek }
\usepackage{stmaryrd}
\SetSymbolFont{stmry}{bold}{U}{stmry}{m}{n}
\usepackage{amsthm}
\usepackage{float}

\usepackage{ bbm }
\usepackage{ stmaryrd }
\usepackage{ mathrsfs }
\usepackage{ frcursive }
\usepackage{ comment }

\usepackage{pgf, tikz}
\usetikzlibrary{shapes}
\usepackage{varioref}
\usepackage{enumitem}

\definecolor{rouge}{rgb}{0.7,0.00,0.00}
\definecolor{vert}{rgb}{0.00,0.5,0.00}
\definecolor{bleu}{rgb}{0.00,0.00,0.8}
\usepackage[margin=1in]{geometry}
\newtheorem{theorem}{Theorem}[section]
\newtheorem*{theorem*}{Theorem}
\newtheorem{lemma}[theorem]{Lemma}

\newtheorem{proposition}[theorem]{Proposition}

\labelformat{hypothesis}{\textbf{M\kern-0.1mm#1}}

\newtheorem{condition}{Condition}

\newtheorem{conditionA}{A\kern-0.1mm}
\labelformat{conditionA}{\textbf{A\kern-0.1mm#1}}

\theoremstyle{definition}

\newtheorem{remark}[theorem]{Remark}

\def \eref#1{\hbox{(\ref{#1})}}

\numberwithin{equation}{section}

\def\geq{\geqslant}
\def\leq{\leqslant}

\def\RR{\mathbb{R}}
\def\PP{\mathbb{P}}
\def\EE{\mathbb{E}}

\def\vare{{\varepsilon}}
\def \eref#1{\hbox{(\ref{#1})}}

\def\EE{\mathbb{ E}}

\begin{document}

\title[Averaging principle for stochastic real Ginzburg-Landau
equation]
{ Averaging principle for stochastic real Ginzburg-Landau equation driven by $\alpha$-stable process}

\author{Xiaobin Sun}
\curraddr[Sun, X.]{ School of Mathematics and Statistics, Jiangsu Normal University, Xuzhou, 221116, China}
\email{xbsun@jsnu.edu.cn}


\author{Jianliang Zhai}
\curraddr[Zhai, J.]{ School of Mathematical Science, University of Science and Technology of China, Hefei, 230026, China.}
\email{zhaijl@ustc.edu.cn}

\begin{abstract}
In this paper, we study a system of stochastic partial differential equations with slow and fast time-scales, where the slow component is a stochastic real Ginzburg-Landau equation and the fast component is a stochastic reaction-diffusion equation, the system is driven by cylindrical $\alpha$-stable process with $\alpha\in (1,2)$. Using the classical Khasminskii approach based on time discretization and the techniques of stopping times, we show that the slow component strong converges to the solution of the corresponding averaged equation under some suitable conditions.
\end{abstract}

\date{\today}

\maketitle

\section{Introduction}

The real Ginzburg-Landau equation is an important model equation in modern physics, and it was broadly studied in recent years from different points of view.
In this paper, we are interested in studying the averaging principle for stochastic real Ginzburg-Landau
equation driven by cylindrical $\alpha$-stable process, i.e., considering the following stochastic slow-fast system on torus $\mathbb{T}=\mathbb{R}/\mathbb{Z}$:
\begin{equation}\left\{\begin{array}{l}\label{Equation}
\displaystyle
d X^{\varepsilon}_t(\xi)=\big[\frac{\partial^2}{\partial \xi^2} X^{\varepsilon}_t(\xi)+X^{\varepsilon}_t(\xi)-(X^{\varepsilon}_t(\xi))^3+f(X^{\varepsilon}_t(\xi), Y^{\varepsilon}_t(\xi))\big]dt+dL_t,\quad X^{\varepsilon}_0(\xi)=x(\xi), \\
d Y^{\varepsilon}_t(\xi)
=\frac{1}{\varepsilon}\big[\frac{\partial^2}{\partial \xi^2} Y^{\varepsilon}_t(\xi)+g(X^{\varepsilon}_t(\xi), Y^{\varepsilon}_t(\xi))\big]dt+\frac{1}{\varepsilon^{1/\alpha}}dZ_t,\quad Y^{\varepsilon}_0(\xi)=y(\xi),\end{array}\right.
\end{equation}
where $\varepsilon >0$ is a small parameter
describing the ratio of time scales between the slow component $X^{\varepsilon}$
and fast component $Y^{\varepsilon}$.
The coefficients $f$ and $g$ are proper functions. $\{L_t\}_{t\geq 0}$ and  $\{Z_t\}_{t\geq 0}$ are mutually independent cylindrical $\alpha$-stable process with $\alpha\in (1,2).$

\vspace{0.1cm}
The theory of averaging principle has a long and rich history  in multiscale model, which has widely applications in nonlinear oscillations, chemical kinetics, biology, climate dynamics and many other areas. Averaging principle is essential to describe the asymptotic behavior of the system like (\ref{Equation}) for $\vare<<1$, and how dynamics of this kind of system depends on $\vare$ as $\vare\rightarrow 0$.  Bogoliubov and Mitropolsky \cite{BM} first studied the averaging principle for the deterministic systems. The averaging principle for stochastic differential equations was proved by Khasminskii \cite{K1}, see, e.g., \cite{B1,C1,CF,C2,FL,FLL,FLWL,G,GJ,GL,L,LRSX,PXY,WR,WRD,XPW} for further generalization. In most of the
existing literature,  Wiener noises are considered. For the jump-diffusion case, we refer to \cite{BYY, Givon, XML}.

\vspace{0.1cm}


It is well known that it is non-trival to deal with nonlinear terms. Recently, averaging principle for stochastic reaction-diffusion systems with nonlinear term has become an active research area which attracted much attention, see, for instance, stochastic reaction-diffusion equations with polynomial coefficients \cite{C2}, stochastic Burgers equation \cite{DSXZ}, stochastic two dimensional Navier-Stokes equations \cite{LSXZ},  stochastic Kuramoto-Sivashinsky equation \cite{GP}, stochastic Schr\"{o}dinger equation \cite{GP1} and stochastic Klein-Gordon equation \cite{GP2}. The noise considered in above references are all Wiener noise, and it is natural to study the stochastic systems driven by jump noise. Averaging principle for stochastic reaction-diffusion systems with nonlinear term driven by jump noise have not yet been considered to date to the best of our knowledge, as a consequence, this is one of the motivations for studying the system (\ref{Equation}).

\vspace{0.1cm}

The noise processes considered to date are mainly square integrable processes. However, a cylindrical/additive $\alpha$-stable process only has finite $p$th moment for $p\in(0,\alpha)$. The methods developed in the existing literature are not suitable to treat the case of cylindrical/additive $\alpha$-stable noises, therefore we require new and different techniques to deal with the cylindrical/additive $\alpha$-stable noise more carefully. Let us mention that Bao et al. \cite{BYY} have studied the averaging principle for two-time scale stochastic partial differential equations driven by additive $\alpha$-stable noises, but without the nonlinear term.
The most challenge in this paper is how to deal with the nonlinear term and cylindrical $\alpha$-stable noises. The techniques of stopping times will be used frequently.

\vspace{0.1cm}

Considering the cylindrical/additive $\alpha$-stable noises has theoretically interesting. As stated in \cite{BYY}, \emph{stochastic equations driven by (cylindrical/additive) $\alpha$-stable processes have proven
to have numerous applications in physics because such processes can be used to model systems
with heavy tails. As a result, such processes have received increasing attentions recently. } We refer to \cite{DXZ1,Ouyang,PZ,Sato,Wang Fengyu} and references therein  for more results and backgrounds, in which the  cylindrical/additive $\alpha$-stable noises have been considered.


\vspace{0.1cm}
 Under some assumptions, the aim of this paper is to prove the slow component $X^{\varepsilon}$ convergent to $\bar{X}$ in the strong sense, i.e., for any $T>0$ and $1\leq p<\alpha$,
\begin{eqnarray}
\mathbb{E}\left(\sup_{0\leq t\leq T}\|X^{\varepsilon}_t-\bar{X}_t\|^{p}_{L^2(\mathbb{T})}\right)\leq \frac{C_{p,k,T}(1+\|x\|^p_{\theta}+\|y\|^{p/2})}{\left(-\ln \vare\right)^{k}},\quad \forall k\in (0, \frac{\alpha-p}{6\alpha}), \label{main result}
\end{eqnarray}
where $(x,y)$ are initial value and belong to some certain spaces, $C_{p,k,T}$ is a positive constant only depends on $p,k,T$, and $\bar{X}$ is the solution of the corresponding averaged equation (see  Eq. \eref{1.3} below).

\vspace{0.1cm}
The proof of our main result is divided into several steps. Firstly, we follow the skills in \cite{Xu1} to give a-priori estimate of the solution $(X^{\varepsilon}_t, Y^{\varepsilon}_t)$, which is very important to constrict stopping times later. Meanwhile, we prove an estimate of $|X_{t}^{\varepsilon}-X_{s}^{\varepsilon}|$ when $s, t$ before the stopping time. Secondly, based on the Khasminskii discretization introduced in \cite{K1}, we split the interval $[0,T]$ into some subintervals of size $\delta>0$ which depends on $\varepsilon$, and on each interval $[k\delta, (k+1)\delta)]$, $k\geq 0$, we construct an auxiliary process $(\hat{X}_t^\varepsilon, \hat{Y}_t^\varepsilon)$ which associate with the system (\ref{Equation}). Finally, by controlling the difference processes $X_t^\varepsilon-\hat{X}_t^\varepsilon$ and $\hat{X}_t^\varepsilon-\bar{X}_t$ respectively, we get an estimate of $|X_t^\varepsilon-\bar{X}_t|$ when time $t$ before the stopping time. Moreover, we use the a-priori estimates of the $X^{\varepsilon}_t$ and $\bar X_t$  are used to control $|X_t^\varepsilon-\bar{X}_t|$ when time $t$ after the stopping time.

\vspace{0.1cm}
The paper is organized as follows. In the next section, we introduce some notation
and assumptions that we use throughout the paper and state the main result. The section 3 is devoted to proving the main result. The final section is the appendix, where we show the detailed proof of the existence and uniqueness of solution and the corresponding Galerkin approximation.

\vspace{0.1cm}
Along the paper, $C$, $C_p$, $C_{T}$ and $C_{p,T}$ denote positive constants which may change from line to line, where $C_p$ depends on $p$, $C_T$ depends on $T$, and $C_{p,T}$ depends on $p,T$.

\section{Notations and main results} \label{Sec Main Result}


For $p\geq 1$ and $k\in\mathbb{N}$, let $L^p(\mathbb{T})$ and $W^{k,2}(\mathbb{T})$ be the usual Sobolev space. Denote $|\cdot|_{L^p}$ the usual
norm of $L^p(\mathbb{T})$.
Set
$$H: = \{x\in L^2(\mathbb{T}): \int_{\mathbb{T}}x(\xi)d\xi=0\},$$
and equipped it with the following inner product and norm:
$$
\langle x, y\rangle:=\int_{\mathbb{T}}x(\xi)y(\xi)d\xi,\quad \|x\|:=\left(\int_{\mathbb{T}}x^2(\xi)d\xi\right)^{1/2},\quad x,y\in H.
$$
Then $(H,\|\cdot\|)$ is a separable real Hilbert space.

Denote the Laplacian operator $\Delta$ by
\begin{align*}
Ax := \Delta x := \frac{\partial^2}{\partial \xi^2}x,\quad x\in W^{k,2}(\mathbb{T})\cap H.
\end{align*}
The eigenfunctions of $\Delta$ are given by
$$
\{e_i: e_i=\sqrt{2}\cos(2\pi i\xi)\}_{i\in \{1,2,\ldots\}}\cup \{e_j: e_j=\sqrt{2}\sin(2\pi j\xi)\}_{j\in \{-1,-2,\ldots\}},
$$
which forms an orthonormal basis of $H$. For any $k\in \mathbb{Z}_{\ast}:=\mathbb{Z}\setminus \{0\}$,
$$Ae_k=-\lambda_ke_k \quad \text{with} \quad \lambda_k = 4\pi^2 |k|^2.$$

For any $s\in\RR$, we define
 $$H^s:=\mathscr{D}((-A)^{s/2}):=\left\{u=\sum_{k\in\mathbb{Z}_{\ast}}u_ke_k: u_k=\langle u,e_k\rangle\in \mathbb{R},~\sum_{k\in\mathbb{Z}_{\ast}}\lambda_k^{s}u_k^2<\infty\right\},$$
 and
 $$(-A)^{s/2}u:=\sum_{k\in\mathbb{Z}_{\ast}}\lambda_k^{s/2} u_ke_k,~~u\in\mathscr{D}((-A)^{s/2}),$$
with the associated norm
\begin{eqnarray*}
\|u\|_{s}:=\|(-A)^{s/2}u\|=\sqrt{\sum_{k\in\mathbb{Z}_{\ast}}\lambda_k^{s} u^2_k}.
\end{eqnarray*}
It is easy to see that $H^{0}=H$ and  $H^{-s}$ is the dual space of $H^s$. Without danger of confusion  the dual action is also denoted by $\langle\cdot,\cdot\rangle$. Denote $V:=H^{1}$, which is densely and compactly embedded in $H$. It is well known that $A$ is the infinitesimal generator of a strongly continuous semigroup
$\{e^{t A}\}_{t\geq0}$.


Define nonliear operator,
\begin{eqnarray*}
N(x)(\xi)=x(\xi)-[x(\xi)]^3,\quad x\in H.
\end{eqnarray*}

We shall often use the following inequalities:

\begin{eqnarray}
&&\| x\|_{\sigma_1} \leq C_{\sigma_1,\sigma_2}\|x\|_{\sigma_2},\quad x\in H^{\sigma_2},\sigma_1\leq\sigma_2;\label{P-1}\\
&&\| e^{tA}x-x \| \leq C_{\sigma} t^{\frac{\sigma}{2}}\|x\|_{\sigma},\quad x\in H^{\sigma},\sigma\geq0, t>0;\label{P0}\\
&&\|e^{tA}x\|\leq e^{-\lambda_1 t}\|x\|,\quad x\in H,  t\geq0;\label{P1/2}\\
&&\|e^{tA}x\|_{\sigma_2}\leq C_{\sigma_1,\sigma_2}t^{-\frac{\sigma_2-\sigma_1}{2}}\|x\|_{\sigma_1},\quad x\in H^{\sigma_2},\sigma_1\leq\sigma_2, t>0;\label{P1}\\
&&\|N(x)\|_{-\sigma}\leq C_{\sigma}(1+\|x\|^{3}_{\frac{1-\sigma}{3}}),\quad x\in H^{\frac{1-\sigma}{3}}, \sigma\in[0,1/2);\label{P2}\\
&&\|N(x)-N(y)\|\leq C(1+\|x\|^2_{1/2}+\|y\|^2_{1/2})\|x-y\|,\quad x,y\in H^{1/2};\label{P3}\\
&& \|N(x)-N(y)\|\leq C(1+\|x\|^2_{2\sigma}+\|y\|^2_{2\sigma})\|x-y\|_{2\sigma},\quad x,y\in H^{2\sigma},\sigma\geq 1/6; \label{P4}\\
&& \langle x, N(x)\rangle\leq \frac{1}{4}, \quad x\in H.\label{P5}
\end{eqnarray}
The properties of \eref{P-1}-\eref{P1} are well known, the proof of \eref{P3}-\eref{P5} can be founded in \cite[Appendix]{Xu1} and the proof of  \eref{P2} is stated in the appendix.

\vspace{0.3cm}
With the above notations, the system \eref{Equation} can be rewritten as:
\begin{equation}\left\{\begin{array}{l}\label{main equation}
\displaystyle
dX^{\vare}_t=\left[AX^{\vare}_t+N(X^{\vare}_t)+f(X^{\vare}_t, Y^{\vare}_t)\right]dt+dL_t,\quad X^{\vare}_0=x,\\
dY^{\vare}_t=\frac{1}{\vare}[AY^{\vare}_t+g(X^{\vare}_t, Y^{\vare}_t)]dt+\frac{1}{\vare^{1/\alpha}}dZ_t,\quad Y^{\vare}_0=y,\end{array}\right.
\end{equation}
where $\{L_t\}_{t\geq 0}$ and  $\{Z_t\}_{t\geq 0}$ are mutually independent cylindrical $\alpha$-stable process given by
$$
L_t=\sum_{k\in \mathbb{Z}_{\ast}}\beta_{k}L^{k}_{t}e_k,\quad Z_t=\sum_{k\in \mathbb{Z}_{\ast}}\gamma_{k}Z^{k}_{t}e_k,\quad t\geq 0,
$$
where $\alpha\in(1,2)$, $\{\beta_k\}_{k\in \mathbb{Z}_{\ast}}$ and $\{\gamma_k\}_{k\in\mathbb{Z}_{\ast}}$ are two given sequence of positive numbers and
$\{L^{k}\}_{k\in \mathbb{Z}_{\ast}}$ and $\{Z^{k}\}_{k\in \mathbb{Z}_{\ast}}$ are independent one dimensional $\alpha$-stable processes satisfying for any $k\in \mathbb{Z}_{\ast}$ and $t\geq0$,
$$\mathbb{E}[e^{i L^k_{t}h}]=\mathbb{E}[e^{i Z^k_{t}h}]=e^{-t|h|^{\alpha}}, \quad h\in \mathbb{R}.$$

\vspace{0.3cm}
We impose the following assumptions.
\begin{conditionA}\label{A1}
$f, g: H\times H \rightarrow H$ are Lipschitz continuous, i.e., there exist constants $C>0$ and $L_{f}, L_{g}>0$ such that for any $x_1,x_2,y_1,y_2\in H$,
\begin{align*}
&&\|f(x_1, y_1)-f(x_2, y_2)\|\leq L_{f}(\|x_1-x_2\| + \|y_1-y_2\|),\\
&&\|g(x_1, y_1)-g(x_2, y_2)\|\leq C\|x_1-x_2\| + L_g\|y_1-y_2\|.
\end{align*}
\end{conditionA}

\begin{conditionA}\label{A2}
There exist constants $C_1, C_2>0$ such that, for any $k\in\mathbb{Z}_{\ast}$,
$$
C_1\lambda_k^{-\beta}\leq \beta_k\leq C_2\lambda_k^{-\beta},\quad \text{with} \quad  \beta>\frac{1}{2}+\frac{1}{2\alpha}
$$
and
\begin{eqnarray}
\sum_{k\in \mathbb{Z}_{\ast}}\frac{\gamma_k^{\alpha}}{\lambda_k}<\infty. \label{I2.10}
\end{eqnarray}
\end{conditionA}

\begin{conditionA}\label{A3}
$f$ is uniformly bounded, i.e.,there exists $C>0$ such that
\begin{align*}
\sup_{x,y\in H}\|f(x, y)\|\leq C.
\end{align*}
\end{conditionA}

\begin{conditionA}\label{A4}
The smallest eigenvalue $\lambda_1$ of $-A$ and the Lipschitz constant $L_g$ satisfy
$$
\lambda_{1}-L_{g}>0.
$$
\end{conditionA}

\begin{remark}\label{Re1}
Under the condition \ref{A1}-\ref{A3}, for any given initial value $x, y\in H$, $\vare>0$, system \eqref{main equation} exists a unique mild
solution $X^{\varepsilon} \in D([0,\infty); H) \cap D((0, \infty); V)$, $Y^{\varepsilon}_t\in H$ (see Theorem \ref{EUS} below). By \cite{Liu-Zhai} and \cite{PZ},
in general, \eref{I2.10} does not imply that $Y^{\varepsilon}\in D([0,\infty); H)$, only implies $Y^{\varepsilon}_t\in H$, but it is enough for us to prove our main result.
Condition \ref{A4} is called the dissipative condition, which is used to give the uniform estimate of $Y^{\vare}$ with respect to $\vare$ and the exponential ergodicity of frozen equation (see Proposition \ref{ergodicity} below).
\end{remark}
\begin{remark}\label{Re2} Define
$$L_A(t):=\int^t_0 e^{(t-s)A}d L_s,\quad Z_A(t):=\int^t_0 e^{(t-s)A}d Z_s.$$
Refer to \cite[Lemma 3.1]{Xu1} and \cite[(4.12)]{PZ}, if condition \ref{A2} holds, then for any $T > 0$, $0 \leq\theta< \beta-\frac{1}{2\alpha}$ and $0<p<\alpha$,
\begin{eqnarray*}
\EE \left[\sup_{0\leq t\leq T}\|L_A(t)\|^p_{2\theta}\right]\leq C_1T^{p/\alpha},\quad \sup_{t\geq 0}\EE\|Z_A(t)\|^p\leq C_2,
\end{eqnarray*}
where $C_1$ depends on $\alpha,\theta,\beta,p$ and $C_2$ depends on $\alpha,p$.
\end{remark}

\medskip
%

Based on the Banach fixed point theorem, we have the following existence and uniqueness of the mild solution of system \eqref{main equation}, whose proof is stated in the appendix.

\begin{theorem}\label{EUS} Assume the conditions \ref{A1}-\ref{A3} hold. Then for every $\varepsilon>0$, $x\in H$, $y\in H$, system \eref{main equation} admits a unique mild solution $X^{\vare}.(\omega)\in D([0,\infty);H)\cap D((0,\infty);V)$ and $Y^{\vare}_t(\omega)\in H$, $t\geq 0$, $\PP$-a.s.. Moreover, for any $t\geq 0$
\begin{equation}\left\{\begin{array}{l}\label{A mild solution}
\displaystyle
X^{\varepsilon}_t=e^{tA}x+\int^t_0e^{(t-s)A}N(X^{\varepsilon}_s)ds+\int^t_0e^{(t-s)A}f(X^{\varepsilon}_s, Y^{\varepsilon}_s)ds+\int^t_0 e^{(t-s)A}dL_s,\\
Y^{\varepsilon}_t=e^{tA/\varepsilon}y+\frac{1}{\varepsilon}\int^t_0e^{(t-s)A/\varepsilon}g(X^{\varepsilon}_s,Y^{\varepsilon}_s)ds
+\frac{1}{\vare^{1/\alpha}}\int^t_0 e^{(t-s)A/\varepsilon}dZ_s.
\end{array}\right.
\end{equation}
\end{theorem}

\vspace{0.3cm}
The main result of this paper is the following theorem.
\begin{theorem}\label{main result 1}
Assume the conditions \ref{A1}-\ref{A4} hold.
Then for any $x\in H^{\theta}$ with $\theta\in(1/2, 1]$,
$y\in H$, $T>0$ and $1\leq p<\alpha$,
\begin{align}
\mathbb{E} \left(\sup_{0\leq t\leq T} \|X_{t}^{\vare}-\bar{X}_{t}\|^{p} \right)\leq \frac{C_{p,k,T}(1+\|x\|^p_{\theta}+\|y\|^{p/2})}{\left(-\ln \vare\right)^{k}},\quad \forall k\in (0, \frac{\alpha-p}{6\alpha}), \label{2.6}
\end{align}
where $C_{p,k,T}$ is a positive constant and $\bar{X}$ is the solution of the corresponding averaged equation:
\begin{equation}\left\{\begin{array}{l}
\displaystyle d\bar{X}_{t}=\left[A\bar{X}_{t}+N(\bar{X}_t)+\bar{f}(\bar{X}_{t})\right]dt+d L_{t},\\
\bar{X}_{0}=x,\end{array}\right. \label{1.3}
\end{equation}
with the average $\bar{f}(x)=\int_{H}f(x,y)\mu^{x}(dy)$. Here $\mu^{x}$ is the unique invariant measure of the frozen equation
\begin{eqnarray*}
\left\{ \begin{aligned}
&dY_{t}=\left[AY_{t}+g(x,Y_{t})\right]dt+d \bar{Z}_{t},\\
&Y_{0}=y,
\end{aligned} \right.
\end{eqnarray*}
$\bar{Z}$ is a version of $Z$ and independent of $L$ and $Z$.
\end{theorem}

\section{Proof of Theorem \ref{main result 1}} \label{Sec Proof of Thm1}

In this section, we are devoted to proving Theorem \ref{main result 1}.
The proof consists of the following several steps. In the first step, we give a priori estimate of the solution $(X^{\varepsilon}_t, Y^{\varepsilon}_t)$ in Lemma \ref{PMY}, which is used to construct a stopping time $\tau^{\vare}_R$. Then Lemma \ref{SOX} gives a uniform estimate of $\|X_{t}^{\varepsilon}\|_{\theta}$ when $t\leq T\wedge \tau^{\varepsilon}_R$ for $\theta\in(1/2, 1]$, which is used to obtain an estimate of the expectation of $X_{t}^{\varepsilon}-X_{s}^{\varepsilon}$ when $0\leq s\leq t\leq T\wedge\tau^{\vare}_R$ in Lemma \ref{COX}. In the second step, following the idea inspired by Khasminskii in \cite{K1}, we introduce an auxiliary process $(\hat{X}_{t}^{\varepsilon},\hat{Y}_{t}^{\varepsilon})\in H \times H$ and also give the uniform bounds, see Lemma \ref{MDY}. Meanwhile, we introduce a new stopping time $\tilde{\tau}^{\vare}_R\leq \tau^{\vare}_R$. Then Lemma \ref{DEY} is used to deduce an estimate of the difference process $X^{\varepsilon}_t-\hat{X}_{t}^{\varepsilon}$ when $t\leq T\wedge \tilde{\tau}^{\vare}_R$, which will be stated in Lemma \ref{DEX}. In the third step, we study the frozen equation and average equation. After defining another stopping time $\hat{\tau}^{\vare}_R\leq\tilde{\tau}^{\vare}_R$, we give a control of $\hat{X}_{t}^{\varepsilon}-\bar{X}_{t}$ when $t\leq T\wedge \hat{\tau}^{\vare}_R$ in Lemma \ref{ESX}. Finally, in order to prove the main result, it is sufficient to control the term of time after the stopping $\hat{\tau}^{\vare}_R$, which will be done by the a-priori estimates of processes $X^{\varepsilon}_t$ and $\bar X_t$ (see Lemmas \ref{PMY} and \ref{barX} below). Note that we always assume conditions \ref{A1}-\ref{A4} hold in this section.

 \subsection{Some a-priori estimates of \texorpdfstring{$(X^{\varepsilon}_t, Y^{\varepsilon}_t)$} {Lg}}

We first prove the uniform bounds for $p$-moment of the solutions $X_{t}^{\varepsilon}$ and $Y_{t}^{\varepsilon}$ for the system \eref{main equation}, with respect to $\varepsilon \in (0,1)$ and $t \in [0,T]$.
The main proof follows the techniques used in \cite{DXZ}, \cite{DXZ1} and \cite{Xu1}, where the authors deal with the 2D stochastic Navier-Stokes equation, 1D stochastic Burgers' equation and stochastic real Ginzburg-Landau equation driven by $\alpha$-stable noise, respectively.

Inspired by the above references, we first have a fast review about the purely jump L\'evy process as following.
$\{L^{k}\}_{k\in\mathbb{Z}_{\ast}}$ are independent one dimensional $\alpha$-stable processes, so they are purely jump L\'evy processes and have the same characteristic function, i.e.,
$$\mathbb{E}e^{i\xi L^{k}_{t}}=e^{t\psi(\xi)}, \quad \forall t>0, k\in\mathbb{Z}_{\ast},$$
$\psi(\xi)$ is a complex valued function called L\'evy symbol given by
$$\psi(\xi)=\int_{\mathbb{R}\setminus\{0\}}(e^{i\xi y}-1-i\xi y 1_{\{|y|_\leq 1\}})\nu(dy),$$
$\nu(dx)=\frac{c}{|x|^{1+\alpha}}dx$ is the L\'evy measure with $c>0$ and satisfies $\int_{\mathbb{R}\setminus\{0\}}1\wedge|y|^{2}\nu(dy)<\infty.$

For $t>0$ and $\Gamma\in\mathscr{B}(\mathbb{R}\setminus\{0\})$, the Poisson random measure associated with $L^{k}$ is defined by
$$N^{k}(t,\Gamma)=\sum_{s\leq t}1_{\Gamma}(L^{k}_s-L^{k}_{s-}),$$
and the corresponding compensated Poisson measure is given by
$$\widetilde{N}^{k}(t,\Gamma)=N^{k}(t,\Gamma)-t\nu(\Gamma).$$
By L\'evy-It\^o's decomposition, one has
$$L^{k}_{t}=\int_{|x|\leq1}x\widetilde{N}^{k}(t,dx)+\int_{|x|>1}x N^{k}(t,dx).$$

\begin{lemma} \label{PMY}
For any $x,y\in H$, $1\leq p<\alpha$ and $T>0$, there exist constants $C_{p,T},\ C_{T}>0$ such that for all $\vare\in(0,1)$,
\begin{align} \label{Control X}
\mathbb{E}\left(\sup_{0\leq t\leq T}\|X_{t}^{\varepsilon} \|^p \right)+\EE\int^T_0 \frac{\|X_{t}^{\varepsilon} \|^2_1}{(\|X_{t}^{\varepsilon} \|^2+1)^{1-p/2}}dt
\leq  C_{p,T}(1+ \|x\| ^p),
\end{align}
\begin{align} \label{Control Y}
\sup_{0\leq t\leq T}
\mathbb{E} \|Y_{t}^{\varepsilon} \|
\leq C_{T}(1+ \|x\|  + \|y\| ).
\end{align}
\end{lemma}

\begin{proof}
%

For $m\in\mathbb{N}_{\ast}$, put $H_m={\rm span}\{e_k, |k|\leq m\}$ and let $\pi_m$ be the projection from $H$ to $H_m$. Consider the Galerkin approximation of system \eref{main equation}:
\begin{equation}\left\{\begin{array}{l}\label{A4.12}
\displaystyle
dX^{m,\vare}_t=[AX^{m,\vare}_t+N^m(X^{m,\vare}_t)+f^m(X^{m,\vare}_t, Y^{m,\vare}_t)]dt+d\bar L^m_t,\quad X^{m,\vare}_0=x^{m}\in H_m\\
dY^{m,\vare}_t=\frac{1}{\vare}[AY^{m,\vare}_t+g^m(X^{m,\vare}_t, Y^{m,\vare}_t)]dt+\frac{1}{\vare^{1/\alpha}}d\bar Z^m_t,\quad Y^{m,\vare}_0=y^m\in H_m,\end{array}\right.
\end{equation}
where $X^{m,\vare}_0=\pi_{m}X^{\vare}_0$, $N^{m}(X^{m,\vare}_t)=\pi_m(N(X^{m,\vare}_t))$, $f^{m}(X^{m,\vare}_t,Y^{m,\vare}_t )=\pi_m(f(X^{m,\vare}_t, Y^{m,\vare}_t))$, $g^{m}(X^{m,\vare}_t,Y^{m,\vare}_t )=\pi_m(g(X^{m,\vare}_t, Y^{m,\vare}_t))$, $\bar L^{m}_t=\sum_{|k|\leq m}\beta_{k}L^{k}_{t}e_k$ and $\bar Z^{m}_t=\sum_{|k|\leq m}\gamma_{k}Z^{k}_{t}e_k$.

Now, define a smooth function $U$ on $H_m$ by
$$U(x)=(\|x\|^{2}+1)^{p/2}, \quad x\in H_m.$$
Then we have
$$
DU(x)=\frac{px}{(\|x\|^2+1)^{1-p/2}},\quad
D^2 U(x)=\frac{p I_m}{(\|x\|^2+1)^{1-p/2}}+\frac{p(p-2)x\otimes x}{(\|x\|^2+1)^{2-p/2}},
$$
where $I_m$ is the union matrix on $H_m$. As a result, for any $x,y\in H_m$,
\begin{eqnarray}
&&\|D U(x)\|\leq \frac{C_p\|x\|}{(\|x\|^2+1)^{1-p/2}}\leq C_p(\|x\|^2+1)^{p/2-1/2},\label{F3.1}\\
&&\|D^2 U(x)\|\leq \frac{C_p}{(\|x\|^2+1)^{1-p/2}}\leq C_p. \label{F3.2}
\end{eqnarray}
By It\^o's formula, we get
\begin{eqnarray}
&&U(X^{m,\vare}_t)+\int^{t}_{0}\frac{ p\|X^{m,\vare}_{s}\|^{2}_1}{(\|X^{m,\vare}_s\|^{2}+1)^{1-p/2}}ds\nonumber\\
&=&
U(x^m)+\int^{t}_{0}\left[\frac{ \langle N^m(X^{m,\vare}_{s}), p X^{m,\vare}_{s}\rangle}{(\|X^{m,\vare}_s\|^{2}+1)^{1-p/2}}+\frac{ \langle f^m(X^{m,\vare}_{s},Y^{m,\vare}_s ), p X^{m,\vare}_{s}\rangle}{(\|X^{m,\vare}_s\|^{2}+1)^{1-p/2}}\right]ds\nonumber\\
&&+\sum_{|k|\leq m}\int^{t}_{0}\int_{|x|\leq1}[U(X^{m,\vare}_{s-}+x\beta_{k}e_k)-U(X^{m,\vare}_{s-})]\widetilde{N}^{k}(ds,dx)\nonumber\\
&&+\sum_{|k|\leq m}\int^{t}_{0}\int_{|x|\leq1}\left[U(X^{m,\vare}_{s}+x\beta_{k}e_k)-U(X^{m,\vare}_{s})
- \frac{\langle p X^{m,\vare}_{s}, x\beta_k e_k\rangle}{(\|X^{m}_s\|^{2}+1)^{1-p/2}}\right]\nu(dx)ds\nonumber\\
&&+\sum_{|k|\leq m}\int^{t}_{0}\int_{|x|>1}[U(X^{m,\vare}_{s-}+x\beta_{k}e_k)-U(X^{m,\vare}_{s-})]N^{k}(ds,dx)\nonumber\\
:=\!\!\!\!\!\!\!\!&&\!\!U(x^m)+I_1^m(t)+I^m_2(t)+I_3^m(t)+I_4^m(t).\label{3.4}
\end{eqnarray}

For $I^{m}_1(t)$, by \eref{P5}, condition \ref{A3}, $p\in[1,\alpha)$ and $\|X^{m,\vare}_s\|^{2}+1\geq 1$, there exists $C_p>0$ such that
\begin{equation}\label{3.4.1}
\EE\left[\sup_{0\leq t\leq T}I^m_1(t)\right]
 \leq
    pT
      +
   C_p\int^{T}_0 \left(\EE\|X^{m,\vare}_s\|^{p}+1\right)ds.
\end{equation}

For $I^{m}_2(t)$, by the Burkholder-Davis-Gundy inequality (cf. \cite[Theorem 1]{57} ) and \eref{F3.1}, we have
\begin{eqnarray}
\mathbb{E}\left[\sup_{0\leq t\leq T}\|I^{m}_2(t)\|\right]
\leq\!\!\!\!\!\!\!\!&& C\sum_{|k|\leq m}\left(\EE\int^T_0\int_{|x|\leq 1} |U(X^{m,\vare}_s+x\beta_k e_k)-U(X^{m,\vare}_s)|^2 \nu(dx)ds\right)^{1/2}\nonumber\\
\leq\!\!\!\!\!\!\!\!&& C\sum_{|k|\leq m}\left(\EE\int^T_0\int_{|x|\leq 1}\int^1_0 \|DU(X^{m,\vare}_s+\xi x\beta_k e_k)\|^2d\xi \|x\beta_k e_k\|^2 \nu(dx)ds\right)^{1/2}\nonumber\\
\leq\!\!\!\!\!\!\!\!&& C\!\!\sum_{|k|\leq m}\left(\EE\int^T_0\int_{|x|\leq 1}\left(\|X^{m,\vare}_s\|^{2p-2}+\|x\beta_k e_k\|^{2p-2}+1\right) \|x\beta_k e_k\|^2 \nu(dx)ds\right)^{1/2}\nonumber\\
\leq\!\!\!\!\!\!\!\!&& C\!\!\sum_{k\in \mathbb{Z}_{\ast}}\beta_k \left[\int_{|x|\leq 1}|x|^{2 }\nu(dx)\right]^{1/2}\!\!\left[\EE\int^T_0\left(\|X^{m,\vare}_s\|^{2p-2}+1\right) ds\right]^{1/2}\nonumber\\
&&+C\sum_{k\in \mathbb{Z}_{\ast}}\beta^{p}_k \left[\int_{|x|\leq 1}|x|^{2p}\nu(dx)\right]^{1/2}\nonumber\\
 \leq
   \!\!\!\!\!\!\!\!&& C\EE\int^T_0\left(\|X^{m,\vare}_s\|^{2p-2}+1\right) ds
   +
   C\Big(
      \sum_{k\in \mathbb{Z}_{\ast}}\beta_k \left[\int_{|x|\leq 1}|x|^{2 }\nu(dx)\right]^{1/2}
   \Big)^2
      \nonumber\\
&&+C\sum_{k\in \mathbb{Z}_{\ast}}\beta^{p}_k \left[\int_{|x|\leq 1}|x|^{2p}\nu(dx)\right]^{1/2}\nonumber\\
\leq
   \!\!\!\!\!\!\!\!&& C\EE\int^T_0\left(\|X^{m,\vare}_s\|^{p}+2\right) ds
   +
   C\Big(
      \sum_{k\in \mathbb{Z}_{\ast}}\beta_k \left[\int_{|x|\leq 1}|x|^{2 }\nu(dx)\right]^{1/2}
   \Big)^2
      \nonumber\\
&&+C\sum_{k\in \mathbb{Z}_{\ast}}\beta^{p}_k \left[\int_{|x|\leq 1}|x|^{2p}\nu(dx)\right]^{1/2},\nonumber\\
\leq
   \!\!\!\!\!\!\!\!&& C\EE\int^T_0\left(\|X^{m,\vare}_s\|^{p}+2\right) ds
   +
   C_{p,T},
 \label{3.7.1}
\end{eqnarray}
where the last two inequalities have used condition \ref{A2} and $\|X^{m,\vare}_s\|^{2p-2}\leq \|X^{m,\vare}_s\|^{p}+1$.

For $I^{m}_3(t)$, the Taylor's expansion and (\ref{F3.2}) imply
\begin{eqnarray}
\mathbb{E}\left[\sup_{0\leq t\leq T}\|I^{m}_3(t)\|\right]\leq C_p\sum_{|k|\leq m}\beta^2_k\int^T_0\int_{|x|\leq 1} |x|^2 \nu(dx)ds\leq C_{p,T}.\label{3.7.2}
\end{eqnarray}

For  $I^{m}_4(t)$, by \eref{F3.1} again, we obtain
\begin{eqnarray}
\mathbb{E}\left[\sup_{0\leq t\leq T}\|I^{m}_4(t)\|\right]\leq\!\!\!\!\!\!\!\!&& C\sum_{|k|\leq m}\EE\left(\int^T_0\int_{|x|> 1} |U(X^{m,\vare}_s+x\beta_k e_k)-U(X^{m,\vare}_s)| N^{k}(ds, dx)\right)\nonumber\\
=\!\!\!\!\!\!\!\!&& C\sum_{|k|\leq m}\EE\left(\int^T_0\int_{|x|>1} |U(X^{m,\vare}_s+x\beta_k e_k)-U(X^{m,\vare}_s)| \nu(dx)ds\right)\nonumber\\
\leq\!\!\!\!\!\!\!\!&& C\sum_{|k|\leq m}\EE\left(\int^T_0\int_{|x|>1} \int^1_0\|DU(X^{m,\vare}_s+\xi x\beta_k e_k)\|d\xi \|x\beta_k e_k\| \nu(dx)ds\right)\nonumber\\
\leq\!\!\!\!\!\!\!\!&& C\sum_{|k|\leq m}\EE\left(\int^T_0\int_{|x|>1}\left(\|X^{m,\vare}_s\|^{p-1}+\|x\beta_k e_k\|^{p-1}+1\right) \|x\beta_k e_k\|  \nu(dx)ds\right)\nonumber\\
\leq\!\!\!\!\!\!\!\!&& C_p\sum_{k\in \mathbb{Z}_{\ast}}\beta_k \int_{|x|>1}|x|\nu(dx)\EE\int^T_0\left(\|X^{m,\vare}_t\|^p+1\right)dt+C_T\sum_{k\in \mathbb{Z}_{\ast}}\beta^{p}_k \int_{|x|>1}|x|^{p}\nu(dx)\nonumber\\
 \leq\!\!\!\!\!\!\!\!&& C_{p,T}\int^T_0\EE\|X^{m,\vare}_t\|^p dt+C_{p,T}.\label{3.7.3}
\end{eqnarray}
Combining estimates \eref{3.4}-\eref{3.7.3}, we have
\begin{eqnarray*}
\mathbb{E}\left[\sup_{0\leq t\leq T}\|X^{m,\vare}_t\|^{p}\right]
+\mathbb{E}\int^{T}_{0}\frac{\|X^{m,\vare}_{t}\|^{2}_1}{(\|X^{m,\vare}_t\|^{2}+1)^{1-p/2}}dt\leq\!\!\!\!\!\!\!\!&&C_{p,T}\int^T_0\EE\|X^{m,\vare}_t\|^p dt+C_{p,T}(1+ \|x\|^{p}).
\end{eqnarray*}
Then the Gronwall's inequality implies
\begin{eqnarray}
\mathbb{E}\left[\sup_{0\leq t\leq T}\|X^{m,\vare}_t\|^{p}\right]
+\mathbb{E}\int^{T}_{0}\frac{\|X^{m,\vare}_{t}\|^{2}_1}{(\|X^{m,\vare}_t\|^{2}+1)^{1-p/2}}dt\leq\!\!\!\!\!\!\!\!&&C_{p,T}(1+ \|x\|^{p}).\label{3.9}
\end{eqnarray}
By Theorem \ref{GA} in the appendix below, for any $t>0$, when $\mathbb{W}=H$ or $V$,
$$\lim_{m\rightarrow\infty}\|X^{m,\vare}_t\|_{\mathbb{W}}=\|X^{\vare}_t\|_{\mathbb{W}},\quad \mathbb{P}-a.s..$$
Hence by Fatou's Lemma and \eref{3.9},
\begin{eqnarray}
\mathbb{E}\left(\sup_{0\leq t\leq T}\|X_{t}^{\varepsilon} \|^{p}\right)+\EE\int^T_0 \frac{\|X_{t}^{\varepsilon} \|^2_1}{(\|X_{t}^{\varepsilon} \|^2+1)^{1-p/2}}dt\label{3.7}
\leq  C_{p,T}(1+ \|x\|^{p}).
\end{eqnarray}

Note that
$$
Y^{\varepsilon}_t=e^{tA/\varepsilon}y+\frac{1}{\varepsilon}\int^t_0e^{(t-s)A/\varepsilon}g(X^{\varepsilon}_s,Y^{\varepsilon}_s)ds
+\frac{1}{\vare^{1/\alpha}}\int^t_0 e^{(t-s)A/\varepsilon}dZ_s.
$$
Then by property \eref{P1/2}, we have for any $t\geq0$,
\begin{eqnarray*}
\|Y^{\varepsilon}_t\|\leq\|y\|+\frac{1}{\varepsilon}\int^t_0e^{-\lambda_1(t-s)/\varepsilon}(C+C\|X^{\varepsilon}_s\|+L_{g}\|Y^{\varepsilon}_s\|)ds
+\left\|\frac{1}{\vare^{1/\alpha}}\int^t_0 e^{(t-s)A/\varepsilon}dZ_s\right\|.
\end{eqnarray*}
Define $\tilde Z_t:=\frac{1}{\vare^{1/ \alpha}}Z_{t\vare}$, which is also a cylindrical $\alpha$-stable process. Then by \cite[(4.12)]{PZ},
\begin{eqnarray*}
\EE\left\|\frac{1}{\vare^{1/\alpha}}\int^t_0 e^{(t-s)A/\varepsilon}dZ_s\right\|=\!\!\!\!\!\!\!\!&&\EE\left\|\int^{t/\vare}_0 e^{(t/\vare-s)A}d\tilde Z_s\right\|\\
\leq\!\!\!\!\!\!\!\!&&C\left(\sum_{k}\gamma^{\alpha}_k \frac{1-e^{-\alpha \lambda_k t/\vare}}{\alpha \lambda_k} \right)^{1/\alpha}\\
\leq\!\!\!\!\!\!\!\!&&C\left(\sum_{k}\frac{\gamma^{\alpha}_k }{\alpha \lambda_k} \right)^{1/\alpha}.
\end{eqnarray*}
Combining this and \eref{3.7}, we have for any $t\leq T$,
\begin{eqnarray*}
\EE\|Y^{\varepsilon}_t\|\leq\!\!\!\!\!\!\!\!&&\|y\|+\frac{C}{\varepsilon}\int^t_0e^{-\lambda_1(t-s)/\varepsilon}ds+\frac{C}{\varepsilon}\int^t_0e^{-\lambda_1(t-s)/\varepsilon}\EE\|X^{\varepsilon}_s\|ds\\
\!\!\!\!\!\!\!\!&&+\frac{L_g}{\varepsilon}\int^t_0e^{-\lambda_1(t-s)/\varepsilon}L_g\EE\|Y^{\varepsilon}_s\|ds
+\frac{1}{\vare^{1/\alpha}}\EE\left\|\int^t_0 e^{(t-s)A/\varepsilon}dZ_s\right\|\\
\leq\!\!\!\!\!\!\!\!&& C_T(1+\|x\|+\|y\|)+\frac{L_{g}}{\lambda_1}\sup_{0\leq t\leq T}\EE\|Y^{\varepsilon}_t\|.
\end{eqnarray*}
Hence \eref{Control Y} holds due to $L_g<\lambda_1$ by condition \ref{A4}. The proof is complete.
\end{proof}

In order to study the high regularity of the slow component $X^{\vare}_t$, we need to construct the following stopping time, i.e., for any $\vare\in(0,1)$, $R>0$,
\begin{eqnarray}
\tau^{\varepsilon}_R:=\inf\{t>0, \|X^{\vare}_t\|\geq R\}.\label{ST}
\end{eqnarray}

\begin{lemma} \label{SOX}
For any $x\in H^{\theta}$ with $\theta\in (1/2,1]$, $y\in H$, $T>0$, $1\leq p<\alpha$, $R>0$ and $\vare\in(0,1)$, there exists a constant $C_{p,T}>0$ such that
\begin{align}
\mathbb{E}\big(\sup_{0\leq t\leq T\wedge \tau^{\varepsilon}_R}\|X_{t}^{\varepsilon}\|_{\theta}^{p}\big)
\leq C_{p,T}e^{C_{p,T}R^6}(\|x\|^p_{\theta}+1).\label{4.5}
\end{align}
\end{lemma}

\begin{proof}
Recall that
\begin{align*}
X^{\varepsilon}_t=e^{tA}x+\int^t_0e^{(t-s)A}N(X^{\varepsilon}_s)ds+\int^t_0e^{(t-s)A}f(X^{\varepsilon}_s, Y^{\varepsilon}_s)ds+\int^t_0 e^{(t-s)A}dL_s.
\end{align*}

According to properties \eref{P1} and \eref{P2}, for any $\theta\in (1/2, 1]$, we have
\begin{eqnarray*}
\|X^{\varepsilon}_t\|_{\theta}\leq\!\!\!\!\!\!\!\!&&\|e^{tA}x\|_{\theta}+\left\|\int^t_0e^{(t-s)A}N(X^{\varepsilon}_s)ds\right\|_{\theta}+\left\|\int^t_0e^{(t-s)A}f(X^{\varepsilon}_s, Y^{\varepsilon}_s)ds\right\|_{\theta}+\left\|\int^t_0 e^{(t-s)A}dL_s\right\|_{\theta}\\
\leq \!\!\!\!\!\!\!\!&&\|x\|_{\theta}+\int^t_0(t-s)^{-1/2}\|N(X^{\varepsilon}_s)\|_{-(1-\theta)}ds+C\int^t_0 (t-s)^{-\theta/2}ds+\|L_{A}(t)\|_{\theta} \nonumber\\
\leq \!\!\!\!\!\!\!\!&& \|x\|_{\theta}+\|L_{A}(t)\|_{\theta}+C_T+C\int^t_0(t-s)^{-1/2}(1+\|X^{\varepsilon}_s\|^3_{\theta/3})ds.
\end{eqnarray*}
Using the interpolation inequality,
$$
\|X^{\varepsilon}_s\|_{\theta/3}
\leq C \| X^{\varepsilon}_s \|^{2/3}\|X^{\varepsilon}_s\|_{\theta}^{1/3}.
$$
Then for any $t\leq T\wedge\tau^{\varepsilon}_R$,
\begin{eqnarray*}
\|X^{\varepsilon}_t\|_{\theta}\leq\!\!\!\!\!\!\!\!&&\|x\|_{\theta}+\|L_{A}(t)\|_{\theta}+C_T+\int^t_0(t-s)^{-1/2}\|X^{\varepsilon}_s\|^2 \|X^{\varepsilon}_s\|_{\theta}ds \nonumber\\
\leq \!\!\!\!\!\!\!\!&&\|x\|_{\theta}+\sup_{0\leq t\leq T}\|L_{A}(t)\|_{\theta}+C_T+\sup_{t< T\wedge\tau^{\varepsilon}_R}\|X^{\varepsilon}_t\|^2\int^t_0(t-s)^{-1/2} \|X^{\varepsilon}_s\|_{\theta}ds,\nonumber\\
\leq \!\!\!\!\!\!\!\!&&\|x\|_{\theta}+\sup_{0\leq t\leq T}\|L_{A}(t)\|_{\theta}+C_T+R^2\int^t_0(t-s)^{-1/2} \|X^{\varepsilon}_s\|_{\theta}ds\nonumber\\
\leq \!\!\!\!\!\!\!\!&&\|x\|_{\theta}+\sup_{0\leq t\leq T}\|L_{A}(t)\|_{\theta}+C_T+R^2\left[\int^t_0(t-s)^{-3/4}ds\right]^{2/3}\left[\int^t_0\|X^{\varepsilon}_s\|^3_{\theta}ds\right]^{1/3},
\end{eqnarray*}
which implies
\begin{eqnarray*}
\sup_{0\leq t\leq T\wedge \tau^{\varepsilon}_R}\|X^{\varepsilon}_t\|^3_{\theta}
\leq \!\!\!\!\!\!\!\!&&C_T\left(\|x\|_{\theta}+\sup_{0\leq t\leq T}\|L_{A}(t)\|_{\theta}+1\right)^3+C_{T}R^6\int^{T\wedge \tau^{\varepsilon}_R}_0\|X^{\varepsilon}_s\|^3_{\theta}ds.
\end{eqnarray*}
Then the Gronwall's inequality yields
\begin{eqnarray*}
\sup_{0\leq t\leq T\wedge \tau^{\varepsilon}_R}\|X^{\varepsilon}_{t}\|_{\theta}\leq \!\!\!\!\!\!\!\!&&C_T e^{C_T R^6}(\|x\|_{\theta}+\sup_{0\leq t\leq T}\|L_{A}(t)\|_{\theta}+1).
\end{eqnarray*}
Hence, by Remark \ref{Re2}, we have for any $1\leq p<\alpha$,
\begin{align} \label{HolderNorm B}
\mathbb{E}\left(\sup_{0\leq t\leq T\wedge \tau^{\varepsilon}_R}\|X^{\varepsilon}_{t}\|^p_{\theta}\right)
\leq C_{p,T}e^{C_{p,T}R^6}(\|x\|^p_{\theta}+1).
\end{align}
The proof is complete.
\end{proof}

\vspace{0.3cm}
Because that we will use the approach based on time discretization later, we first give an estimate of $X_{t+h}^{\varepsilon}-X_{t}^{\varepsilon}$ when $0\leq t\leq t+h\leq T\wedge\tau^{\vare}_R$.
\begin{lemma} \label{COX}
For any $x\in H^{\theta}$ with $\theta\in (1/2,1]$, $y\in H$, $T>0$, $1\leq p<\alpha$ and $R>0$,
there exists a constant $C_{p,T}>0$ such that for all $\vare\in(0,1)$,
\begin{align*}
\mathbb{E}\left[ \| X_{t+h}^{\varepsilon}-X_{t}^{\varepsilon} \|^{p}1_{\{0\leq t\leq t+h\leq T\wedge\tau^{\varepsilon}_{R} \}}\right]
\leq C_{p,T}e^{C_{p,T}R^6}(\|x\|^p_{\theta}+1)h^{\frac{p\theta}{2}}.
\end{align*}
\end{lemma}

\begin{proof}
After simple calculations, it is easy to see
\begin{eqnarray}
X_{t+h}^{\varepsilon}-X_{t}^{\varepsilon}=\!\!\!\!\!\!\!\!&&(e^{Ah}-I)X_{t}^{\varepsilon}+\int_{t}^{t+h}e^{(t+h-s)A}N(X^{\varepsilon}_s)ds \nonumber\\
\!\!\!\!\!\!\!\!&&+\int_{t}^{t+h}e^{(t+h-s)A}f(X^{\varepsilon}_s, Y^{\varepsilon}_s)ds+\int_{t}^{t+h}e^{(t+h-s)A}dL_s  \nonumber\\
:=\!\!\!\!\!\!\!\!&&I_{1}+I_{2}+I_{3}+I_{4}.  \nonumber
\end{eqnarray}

For $I_1$, by \eref{P0} and Lemma \ref{SOX}, for any $1\leq p<\alpha$ we have
\begin{eqnarray}  \label{REGX1}
\mathbb{E} \left[\|I_{1} \|^p 1_{\{0\leq t\leq t+h\leq T\wedge \tau^{\varepsilon}_R\}}\right] \leq\!\!\!\!\!\!\!\!&&
Ch^{\frac{p\theta}{2}}\mathbb{E}\left[\|X^{\varepsilon}_t\|_{\theta}^{p}1_{\{0\leq t\leq T\wedge \tau^{\varepsilon}_R\}}\right]\nonumber\\
\leq\!\!\!\!\!\!\!\!&&C_{p,T}e^{C_{p,T}R^6}(\|x\|^p_{\theta}+1)h^{\frac{p\theta}{2}}.
\end{eqnarray}

For $I_{2}$, by \eref{P2} and interpolation inequality, we get
\begin{eqnarray*}
\|I_{2} \|1_{\{0\leq t\leq t+h\leq\tau^{\varepsilon}_{R}\wedge T\}}\leq\!\!\!\!\!\!\!\!&&C\left[\int_{t}^{t+h} (t+h-s)^{-\frac{1-\theta}{2}}\|N(X^{\varepsilon}_s)\|_{-(1-\theta)}ds\right]1_{\{0\leq t\leq t+h\leq T\wedge \tau^{\varepsilon}_R\}} \nonumber\\
\leq\!\!\!\!\!\!\!\!&&C\left[\int_{t}^{t+h} (t+h-s)^{-\frac{1-\theta}{2}}(1+\|X^{\varepsilon}_s\|^2\|X^{\varepsilon}_s\|_{\theta})ds\right]1_{\{0\leq t\leq t+h\leq T\wedge \tau^{\varepsilon}_R\}} \nonumber\\
\leq\!\!\!\!\!\!\!\!&&Ch^{\frac{1+\theta}{2}}+R^2\sup_{0\leq t\leq T\wedge \tau^{\varepsilon}_R}\|X^{\varepsilon}_t\|_{\theta} h^{\frac{1+\theta}{2}}.
\end{eqnarray*}
Then by Lemma \ref{SOX}, we have
\begin{eqnarray}\label{REGX2}
\EE\left[\|I_{2} \|^p 1_{\{0\leq t\leq t+h\leq T\wedge \tau^{\varepsilon}_R\}}\right]
\leq \!\!\!\!\!\!\!\!&& C_{p,T}R^2 e^{C_{p,T}R^6}(\|x\|^p_{\theta}+1)h^{\frac{p\theta}{2}}\nonumber\\
\leq \!\!\!\!\!\!\!\!&& C_{p,T}e^{C_{p,T}R^6}(\|x\|^p_{\theta}+1)h^{\frac{p\theta}{2}}.
\end{eqnarray}

For $I_{3}$, by condition \ref{A2}, we obtain
\begin{eqnarray}  \label{REGX3}
\mathbb{E} \|I_{3} \|^{p}\leq\!\!\!\!\!\!\!\!&& Ch^{p}.
\end{eqnarray}

For $I_{4}$, Remark \ref{Re2} implies
\begin{eqnarray}\label{REGX4}
\EE \|I_4 \|^{p}\leq\!\!\!\!\!\!\!\!&&C_{p}h^{\frac{p}{\alpha}}.
\end{eqnarray}

Putting \eref{REGX1}-\eref{REGX4} together, which complete the proof.
\end{proof}



\subsection{ Estimates of the auxiliary process
\texorpdfstring{ $(\hat{X}_{t}^{\varepsilon},\hat{Y}_{t}^{\varepsilon})$} {Lg} }

Following the idea inspired by Khasminskii \cite{K1},
we introduce an auxiliary process
$(\hat{X}_{t}^{\varepsilon},\hat{Y}_{t}^{\varepsilon})\in H \times H$.
Specifically, we split the interval $[0,T]$ into some subintervals of size $\delta>0$, where $\delta$ is a positive number depends on $\vare$ and will be chosen later.
With the initial value $\hat{Y}_{0}^{\varepsilon}=Y^{\varepsilon}_{0}=y$,
we construct the process $\hat{Y}_{t}^{\varepsilon}$ as follows:
$$
d\hat{Y}_{t}^{\varepsilon}=\frac{1}{\varepsilon}\left[A\hat{Y}_{t}^{\varepsilon}+g(X_{t(\delta)}^{\varepsilon},\hat{Y}_{t}^{\varepsilon})\right]dt+\frac{1}{\varepsilon^{1/\alpha}}dZ_t,\quad \hat{Y}_{0}^{\varepsilon}=y,
$$
which satisfies
\begin{align} \label{AuxiliaryPro Y 01}
\hat{Y}_{t}^{\varepsilon}=e^{tA/\vare}y+\frac{1}{\varepsilon}\int_{0}^{t}e^{(t-s)A/\vare}g(X_{s(\delta)}^{\varepsilon},\hat{Y}_{s}^{\varepsilon})ds+\frac{1}{\varepsilon^{1/\alpha}}\int_{0}^{t}e^{(t-s)A/\vare}dZ_s,
\end{align}
where $t(\delta)=[\frac{t}{\delta}]\delta$ is the nearest breakpoint proceeding $t$. Then we construct the process $\hat{X}_{t}^{\varepsilon}$
as follows:
$$
d\hat{X}_{t}^{\varepsilon}=\left[A\hat{X}_{t}^{\varepsilon}+N(\hat X_{t}^{\varepsilon})+f(X_{t(\delta)}^{\varepsilon},\hat{Y}_{t}^{\varepsilon})\right]dt+dL_t,\quad \hat{X}_{0}^{\varepsilon}=x,
$$
which satisfies
\begin{align} \label{AuxiliaryPro X 01}
\hat X^{\varepsilon}_t=e^{tA}x+\int^t_0e^{(t-s)A}N(\hat X^{\varepsilon}_s)ds+\int^t_0e^{(t-s)A}f(X^{\varepsilon}_{s(\delta)}, \hat Y^{\varepsilon}_s)ds+\int^t_0 e^{(t-s)A}dL_s.
\end{align}

The following Lemma gives a control of the auxiliary process $(\hat{X}_{t}^{\varepsilon},\hat{Y}_{t}^{\varepsilon})$.
Since the proof almost follows the same steps in the proof of
Lemma \ref{PMY}, we omit the proof here.
\begin{lemma} \label{MDY}
For any $x,y\in H$, $T>0$ and $1\leq p<\alpha$, there exists a constant $C_{p,T}>0$ such that
\begin{align} \label{hatXHolderalpha}
\mathbb{E}\left(\sup_{0\leq t\leq T}\|\hat X_{t}^{\varepsilon} \|^p\right)+\EE\int^T_0 \frac{\|\hat X_{t}^{\varepsilon} \|^2_1}{(\|\hat X_{t}^{\varepsilon} \|^2+1)^{1-p/2}}dt\leq  C_{p,T}(1+ \|x\|^p),
\end{align}
$$
\sup_{0\leq t\leq T}
\mathbb{E} \| \hat{Y}_{t}^{\vare} \|\leq C_{T}(1+ \|x \|+ \|y\|).
$$
\end{lemma}

\begin{lemma} \label{DEY}
For any $x\in H^{\theta}$ with $\theta\in(1/2,1]$, $y\in H$, $1\leq p<\alpha$, $T>0$ and $R>0$, there exists a constant $C_{p,T}>0$ such that
$$
\EE\left(\int^{T\wedge \tau^{\vare}_R}_0\|Y_{t}^{\varepsilon}-\hat Y_{t}^{\varepsilon}\|dt\right)^p
\leq C_{p,T}e^{C_{p,T}R^6}(\|x\|^p_{\theta}+1)\delta^{\frac{p\theta}{2}},
$$
where $\tau^{\vare}_R$ is defined by \eref{ST} .
\end{lemma}
\begin{proof}
By the construction of $Y_{t}^{\varepsilon}$ and $\hat{Y}_{t}^{\varepsilon}$, we have
\begin{eqnarray*}
Y_{t}^{\varepsilon}-\hat Y_{t}^{\varepsilon}=\frac{1}{\vare}\int^t_0 e^{(t-s)A/\vare}\left[g(X_{s}^{\varepsilon},Y_{s}^{\varepsilon})-g(X_{s(\delta)}^{\varepsilon},\hat{Y}_{s}^{\varepsilon})\right]ds.
\end{eqnarray*}
Then for any $t>0$,
\begin{eqnarray*}
\|Y_{t}^{\varepsilon}-\hat Y_{t}^{\varepsilon}\|\leq\!\!\!\!\!\!\!\!&&\frac{1}{\vare}\int^t_0 e^{-\lambda_1(t-s)/\vare}L_g\|X_{s}^{\varepsilon}-X_{s(\delta)}^{\varepsilon}\|ds\nonumber\\
\!\!\!\!\!\!\!\!&&+\frac{1}{\vare}\int^t_0 e^{-\lambda_1(t-s)/\vare}L_g\|Y_{s}^{\varepsilon}-\hat Y_{s}^{\varepsilon}\|ds.
\end{eqnarray*}
By Fubini's theorem,
\begin{eqnarray*}
\int^{T\wedge \tau^{\vare}_R}_0\|Y_{t}^{\varepsilon}-\hat Y_{t}^{\varepsilon}\|dt\leq\!\!\!\!\!\!\!\!&&\frac{1}{\vare}\int^{T\wedge \tau^{\vare}_R}_0\int^t_0 e^{-\lambda_1(t-s)/\vare}L_g\|X_{s}^{\varepsilon}-X_{s(\delta)}^{\varepsilon}\|dsdt\nonumber\\
\!\!\!\!\!\!\!\!&&+\frac{1}{\vare}\int^{T\wedge \tau^{\vare}_R}_0\int^t_0 e^{-\lambda_1(t-s)/\vare}L_g\|Y_{s}^{\varepsilon}-\hat Y_{s}^{\varepsilon}\|dsdt\\
=\!\!\!\!\!\!\!\!&&\frac{L_g}{\vare}\int^{T\wedge \tau^{\vare}_R}_0\left(\int^{T\wedge \tau^{\vare}_R}_s e^{-\lambda_1(t-s)/\vare}dt\right)\|X_{s}^{\varepsilon}-X_{s(\delta)}^{\varepsilon}\|ds\nonumber\\
\!\!\!\!\!\!\!\!&&+\frac{L_g}{\vare}\int^{T\wedge \tau^{\vare}_R}_0\left(\int^{T\wedge \tau^{\vare}_R}_s e^{-\lambda_1(t-s)/\vare}dt\right)\|Y_{s}^{\varepsilon}-\hat Y_{s}^{\varepsilon}\|ds\\
\leq\!\!\!\!\!\!\!\!&&C\int^{T\wedge \tau^{\vare}_R}_0\|X_{s}^{\varepsilon}-X_{s(\delta)}^{\varepsilon}\|ds+\frac{L_g}{\lambda_1}\int^{T\wedge \tau^{\vare}_R}_0\|Y_{s}^{\varepsilon}-\hat Y_{s}^{\varepsilon}\|ds.
\end{eqnarray*}
By Lemma \ref{COX} and $L_g<\lambda_1$, we have
\begin{eqnarray*}
\EE\left(\int^{T\wedge \tau^{\vare}_R}_0\|Y_{t}^{\varepsilon}-\hat Y_{t}^{\varepsilon}\|dt\right)^p
\leq\!\!\!\!\!\!\!\!&&C_{p,T}e^{C_{p,T}R^6}(\|x\|^p_{\theta}+1)\delta^{\frac{p\theta}{2}}.
\end{eqnarray*}
The proof is complete.
\end{proof}

\vspace{0.3cm}
 In the next lemma, we shall deal with the difference process $X^{\vare}_t-\hat{X}_{t}^{\vare}$. To this end, we construct another stopping time, i.e., for any $\vare\in(0,1)$, $R>0$,
\begin{eqnarray*}
\tilde{\tau}^{\varepsilon}_R:=\!\!\!\!\!\!\!\!&&\inf\Big\{t>0:\|X^{\vare}_t\|+\|\hat X^{\vare}_t\|+\int^t_0\frac{\|X^{\vare}_s\|^2_1}{(\|X^{\vare}_s\|^2+1)^{1/2}}ds\\
\!\!\!\!\!\!\!\!&&\quad\quad\quad\quad\quad\quad+\int^t_0\frac{\|\hat X^{\vare}_s\|^2_1}{(\|\hat X^{\vare}_s\|^2+1)^{1/2}}ds\geq R\Big\}.
\end{eqnarray*}

\begin{lemma} \label{DEX}
For any $x\in H^{\theta}$ with $\theta\in (1/2, 1]$, $y\in H$, $1\leq p<\alpha$, $T>0$ and $R>0$
there exists a constant $C_{p,T}>0$ such that
\begin{align*}
\mathbb{E}\Big(\sup_{0\leq t\leq T\wedge \tilde\tau^{\varepsilon}_R} \|X_{t}^{\vare}-\hat{X}_{t}^{\vare} \|^{p}\Big)
\leq C_{p,T}e^{C_{p,T}R^6}(\|x\|^p_{\theta}+1)\delta^{\frac{p\theta}{2}}.
\end{align*}
\end{lemma}

\begin{proof}
In view of \eqref{AuxiliaryPro X 01} and \eqref{A mild solution}, we have
\begin{align*}
X^{\vare}_t-\hat{X}^{\vare}_t=\int^t_0e^{(t-s)A}\big[N(X^{\vare}_s)-N(\hat X^{\vare}_{s})\big]ds
+\int^t_0e^{(t-s)A}\big[f(X^{\vare}_s, Y^{\vare}_s)-f(X^{\vare}_{s(\delta)}, \hat{Y}^{\vare}_s)\big]ds.
\end{align*}

Using condition \ref{A1}, properties \eref{P-1} and \eref{P3}, we get
\begin{eqnarray}
\| X_{t}^{\vare}-\hat{X}_{t}^{\vare} \|
\leq\!\!\!\!\!\!\!\!&&\int_{0}^{t}\big\|N(X^{\vare}_s)-N(\hat X^{\vare}_{s})\big\|ds
+\int_{0}^{t}\|f(X^{\vare}_s, Y^{\vare}_s)-f(X^{\vare}_{s(\delta)}, \hat{Y}^{\vare}_s)\|ds\nonumber\\
\leq\!\!\!\!\!\!\!\!&&C\int_{0}^{t}\big(1+\|X^{\vare}_s\|^2_{1}+\|\hat X^{\vare}_{s}\|^2_{1}\big)\|X^{\vare}_s-\hat X^{\vare}_{s}\|ds\nonumber\\
\!\!\!\!\!\!\!\!&&
+C\int_{0}^{t}\big(\|X^{\vare}_s-X^{\vare}_{s(\delta)}\|+\|Y^{\vare}_s-\hat{Y}^{\vare}_s\|\big)ds.\nonumber
\end{eqnarray}
The Gronwall's inequality implies for any $T>0$,
\begin{eqnarray*}
\sup_{0\leq t\leq T}\| X_{t}^{\vare}-\hat{X}_{t}^{\vare} \|\leq\!\!\!\!\!\!\!\!&&C\int_{0}^{T}\big(\|X^{\vare}_s-X^{\vare}_{s(\delta)}\|+\|Y^{\vare}_s-\hat{Y}^{\vare}_s\|\big)ds e^{C\int^{T}_0 \left(1+\|X^{\vare}_s\|^2_1+\|\hat X^{\vare}_s\|^2_1\right)ds}.
\end{eqnarray*}
Then by the definition of $\tilde{\tau}^{\varepsilon}_R$, we have
\begin{eqnarray*}
&&\sup_{0\leq t\leq T\wedge \tilde{\tau}^{\varepsilon}_R}\| X_{t}^{\vare}-\hat{X}_{t}^{\vare} \|\\
\leq\!\!\!\!\!\!\!\!&&C\int_{0}^{T\wedge \tilde{\tau}^{\varepsilon}_R}\big(\|X^{\vare}_s-X^{\vare}_{s(\delta)}\|+\|Y^{\vare}_s-\hat{Y}^{\vare}_s\|\big)ds e^{C\int_{0}^{T\wedge \tilde{\tau}^{\varepsilon}_R}\left[\frac{\|X^{\vare}_{s}\|^2_{1}(\|X^{\vare}_{s}\|^2+1)^{1/2}}{(\|X^{\vare}_{s}\|^2+1)^{1/2}}+\frac{\|\hat X^{\vare}_{s}\|^2_{1}(\|\hat X^{\vare}_{s}\|^2+1)^{1/2}}{(\|\hat X^{\vare}_{s}\|^2+1)^{1/2}}+1\right]ds}\nonumber\\
\leq\!\!\!\!\!\!\!\!&&C\int_{0}^{T\wedge \tilde{\tau}^{\varepsilon}_R}\big(\|X^{\vare}_s-X^{\vare}_{s(\delta)}\|+\|Y^{\vare}_s-\hat{Y}^{\vare}_s\|\big)ds e^{C(R+1)\int^{T\wedge \tilde{\tau}^{\varepsilon}_R}_0 \left[\frac{\|X^{\vare}_s\|^2_1}{(\|X^{\vare}_s\|^2+1)^{1/2}}+\frac{\|\hat X^{\vare}_s\|^2_1}{(\|\hat X^{\vare}_s\|^2+1)^{1/2}}+1\right]ds}\nonumber\\
\leq\!\!\!\!\!\!\!\!&&Ce^{C_T(R^2+1)}\int_{0}^{T\wedge \tilde{\tau}^{\varepsilon}_R}\big(\|X^{\vare}_s-X^{\vare}_{s(\delta)}\|+\|Y^{\vare}_s-\hat{Y}^{\vare}_s\|\big)ds.
\end{eqnarray*}
Note that $\tilde{\tau}^{\varepsilon}_R\leq \tau^{\varepsilon}_R$, then it follows from Lemmas \ref{COX} and \ref{DEY}, we have
\begin{eqnarray*}
\EE\left[\sup_{0\leq t\leq T\wedge \tilde{\tau}^{\varepsilon}_R}\| X_{t}^{\vare}-\hat{X}_{t}^{\vare} \|^p\right]
\leq C_{p,T}e^{C_{p,T}R^6}(\|x\|^p_{\theta}+1)\delta^{\frac{p\theta}{2}}.\nonumber
\end{eqnarray*}
The proof is complete.  
\end{proof}

\subsection{The frozen and averaged equation}\label{sub 3.3}
For any fixed $x\in H$, we first consider the following frozen equation
associated with the fast component:
\begin{equation} \label{FEQ}
dY_{t}=\left[AY_{t}dt+g(x,Y_{t})\right]dt+d\bar Z_t,\quad Y_{0}=y,
\end{equation}
where $\bar Z_t$ is a version of $Z_t$ and independent of $\{L_t\}_{t\geq 0}$ and  $\{Z_t\}_{t\geq 0}$. Since $g(x,\cdot)$ is Lipshcitz continuous,
it is easy to prove that for any fixed $x, y \in H$,  
the Eq. $\eref{FEQ}$ has a unique mild solution denoted by $Y_{t}^{x,y}$.
For any $x \in H$,
let $P^x_t$ be the transition semigroup of $Y_{t}^{x,y}$,
that is, for any bounded measurable function $\varphi$ on $H$ and $t \geq 0$,
\begin{align*}
P^x_t \varphi(y)= \mathbb{E} \varphi(Y_{t}^{x,y}), \quad y \in H.
\end{align*}
The asymptotic behavior of $P^x_t$ has been studied in many literatures,
the following result shows the existence and uniqueness of the invariant measure and
gives the exponential convergence to the equilibrium (see \cite[Lemma 3.3]{BYY}).
\begin{proposition}\label{ergodicity}
For any $x, y\in H$, $P^x_t$ admits a unique invariant measure $\mu^x$. Moreover, for any $t>0$,
\begin{eqnarray*}
\left\| \mathbb{E} f(x,Y_{t}^{x,y})-\int_{H}f(x,z)\mu^x(dz)\right\|
\leq C(1+ \|x \| + \|y \|)e^{-(\lambda_1-L_g)t},
\end{eqnarray*}
where $C$ is a positive constant which is independent of $t$.
\end{proposition}

The following lemma is used to prove the existence and uniqueness of the solution of corresponding averaged equation, we state it ahead.
\begin{lemma} \label{L3.17} For any $x_1, x_2\in H$, $y\in H$,  we have
\begin{eqnarray*}
\sup_{t\geq 0}\|Y^{x_1,y}_t-Y^{x_2,y}_t\|^2\leq (\lambda_1-L_g)^{-1}\|x_1-x_2\|^2.
\end{eqnarray*}
\end{lemma}
\begin{proof}
Note that
\begin{eqnarray*}
d(Y^{x_1,y}_t-Y^{x_2,y}_t)=A(Y^{x_1,y}_t-Y^{x_2,y}_t)dt+\left[g(x_1, Y^{x_1,y}_t)-g(x_2, Y^{x_2,y}_t)\right]dt.
\end{eqnarray*}
By Young's inequality, it is easy to see
\begin{eqnarray*}
\frac{d}{dt}\|Y^{x_1,y}_t-Y^{x_2,y}_t\|^2=\!\!\!\!\!\!\!\!&&2\|Y^{x_1,y}_t-Y^{x_2,y}_t\|^2_1+2\langle g(x_1, Y^{x_1,y}_t)-g(x_2, Y^{x_2,y}_t), Y^{x_1,y}_t-Y^{x_2,y}_t\rangle\\
\leq\!\!\!\!\!\!\!\!&&-2\lambda_1\|Y^{x_1,y}_t-Y^{x_2,y}_t\|^2\!\!+2L_g \|Y^{x_1,y}_t-Y^{x_2,y}_t\|^2\\
&&+C\|x_1-x_2\|\|Y^{x_1,y}_t-Y^{x_2,y}_t\|\\
\leq\!\!\!\!\!\!\!\!&&-(\lambda_1-L_g)\lambda_1\|Y^{x_1,y}_t-Y^{x_2,y}_t\|^2+C\|x_1-x_2\|^2.
\end{eqnarray*}
Then the compare theorem yields
\begin{eqnarray*}
\sup_{t\geq 0}\|Y^{x_1,y}_t-Y^{x_2,y}_t\|^2\leq \int^{\infty}_{0}e^{-(\lambda_1-L_g)s}ds\|x_1-x_2\|^2\leq (\lambda_1-L_g)^{-1}\|x_1-x_2\|^2.
\end{eqnarray*}
The proof is complete.
\end{proof}

\vspace{0.3cm}
Now, we introduce the averaged equation which satisfies

\begin{equation}
\left\{\begin{array}{l}
\displaystyle d\bar{X}_{t}=\left[A \bar{X}_{t}+N(\bar{X}_{t})+\bar{f}(\bar{X}_{t})\right]dt+dL_t,\\
\bar{X}_{0}=x,\end{array}\right. \label{3.1}
\end{equation}
where
\begin{align*}
\bar{f}(x)=\int_{H}f(x,y)\mu^{x}(dy), \quad x\in H.
\end{align*}

The following lemma is the existence and uniqueness of the solution and its a-priori estimates.
\begin{lemma} \label{barX}
Eq. \eref{3.1} exists a unique mild solution $\bar{X}_{t}$ satisfying
\begin{eqnarray}
\bar X_t=e^{tA}x+\int^t_0e^{(t-s)A}N(\bar X_s)ds+\int^t_0e^{(t-s)A} \bar f(\bar X_s)ds+\int^t_0 e^{(t-s)A}dL_s.\label{3.22}
\end{eqnarray}
Moreover, for any $x\in H$, $T>0$ and $1\leq p< \alpha$, there exists a constant $C_{p,T}>0$ such that
\begin{align} \label{Control X}
\mathbb{E}\left(\sup_{0\leq t\leq T}\|\bar X_{t} \|^p\right)+\EE\int^T_0 \frac{\|\bar X_{t} \|^2_1}{(\|\bar X_{t} \|^2+1)^{1-p/2}}ds
\leq  C_{p,T}(1+ \|x\|^p).
\end{align}
\end{lemma}
\begin{proof}
It is sufficient to check that the $\bar f$ is Lipschitz continuous and bounded, then the results can be easily obtained by following the procedures in Theorem \ref{EUS} and Lemma \ref{PMY}.
Obviously, $\bar f$ is bounded by the boundedness of $f$. It remain to show $\bar f$ is Lipschitz.

In deed, for any $x_1,x_2,y\in H$ and $t>0$, by Proposition \ref{ergodicity} and Lemma \ref{L3.17}, we have
\begin{eqnarray*}
\|\bar{f}(x_1)-\bar{f}(x_2)\|
\leq\!\!\!\!\!\!\!\!&&\left\|\int_{H} f(x_1,z)\mu^{x_1}(dz)-\int_{H} f(x_2,z)\mu^{x_2}(dz)\right\|\\
\leq\!\!\!\!\!\!\!\!&&\left\|\int_{H} f(x_1,z)\mu^{x_1}(dz)-\EE f(x_1, Y^{x_1,y}_t)\right\|+\left\|\EE f(x_1, Y^{x_1,y}_t)-\EE f(x_2, Y^{x_2,y}_t)\right\|\\
\!\!\!\!\!\!\!\!&&+\left\|\EE f(x_2, Y^{x_2,y}_t)-\int_{H} f(x_2,z)\mu^{x_2}(dz)\right\|\\
\leq\!\!\!\!\!\!\!\!&&C(1+\|x_1\|+\|x_2\|+\|y\|)e^{-(\lambda_1-L_g) t}+C\left(\|x_1-x_2\|+\EE\|Y^{x_1,y}_t-Y^{x_2,y}_t\|\right)\\
\leq\!\!\!\!\!\!\!\!&&C(1+\|x_1\|+\|x_2\|+\|y\|)e^{-(\lambda_1-L_g) t}+C\|x_1-x_2\|.
\end{eqnarray*}
Hence, the proof is completed by letting $t\rightarrow \infty$.
\end{proof}

\vspace{0.3cm}
Now, we intend to estimate the difference process $X^{\vare}_t-\hat{X}_{t}^{\vare}$. Similar as the argument in Lemma \ref{DEX}, we further construct a new stopping time, i.e., for any $\vare\in(0,1)$, $R>0$,
\begin{eqnarray*}
\hat {\tau}^{\varepsilon}_{R}:=\!\!\!\!\!\!\!\!&&\inf\Big\{t>0:\|X^{\vare}_t\|+\|\hat X^{\vare}_t\|+\|\bar X_t\|+\int^t_0\frac{\|X^{\vare}_s\|^2_1}{(\|X^{\vare}_s\|^2+1)^{1/2}}ds\\
&&\quad\quad\quad\quad\quad+\int^t_0\frac{\|\bar X_s\|^2_1}{(\|\bar X_s\|^2+1)^{1/2}}ds+\int^t_0\frac{\|\hat X^{\vare}_s\|^2_1}{(\|\hat X^{\vare}_s\|^2+1)^{1/2}}ds\geq R\Big\}.
\end{eqnarray*}
\begin{lemma} \label{ESX}
For any $x\in H^{\theta}$ with $\theta\in(1/2, 1]$, $y\in H$, $1\leq p<\alpha$, $T>0$, $\vare\in (0,1)$ and $R>0$,
there exists a constant $C_{p,T}>0$ such that
\begin{align*}
\mathbb{E}\left(\sup_{0\leq t\leq T\wedge \hat {\tau}^{\varepsilon}_{R}}
\|\hat{X}_{t}^{\vare}-\bar{X}_{t} \|^{p}\right)
\leq
C_{p,T}e^{C_{p,T}R^6}\Big[(\|x\|^p_{\theta}+1)\delta^{\frac{p\theta}{2}}+(1+\|x\|+\|y\|)^{p/2}\frac{\vare^{p/2}}{\delta^{p/2}}\Big].
\end{align*}
\end{lemma}

\begin{proof}
From \eqref{AuxiliaryPro X 01} and \eqref{3.22}, it is easy to see
\begin{eqnarray} \label{F3.19}
\hat{X}_{t}^{\vare}-\bar{X}_{t}
=\!\!\!\!\!\!\!\!&&\int_{0}^{t}e^{(t-s)A}\left[N(\hat{X}_{s}^{\vare})-N(\bar{X}_{s})\right]ds
+\int_{0}^{t}e^{(t-s)A}
\left[f(X_{s(\delta)}^{\vare},\hat{Y}_{s}^{\vare})-\bar{f}(X_{s}^{\vare})\right]ds \nonumber\\
\!\!\!\!\!\!\!\!&&
+\int_{0}^{t}e^{(t-s)A}\left[\bar{f}(X_{s}^{\vare})-\bar{f}(\hat{X}_{s}^{\vare})\right]ds
+\int_{0}^{t}e^{(t-s)A}\left[\bar{f}(\hat{X}_{s}^{\vare})-\bar{f}(\bar{X}_{s})\right]ds   \nonumber\\
:=\!\!\!\!\!\!\!\!&& \sum_{k=1}^{4}J_{k}(t).
\end{eqnarray}

For $J_{1}(t)$, according to \eref{P3}, we have
\begin{eqnarray}  \label{J1}
\sup_{0\leq t\leq T} \|J_{1}(t) \|
\leq\!\!\!\!\!\!\!\!&&C\int_{0}^{T}\big(1+\|\bar X_s\|^2_{1}+\|\hat X^{\vare}_{s}\|^2_{1}\big)\|\hat X^{\vare}_s-\bar X_{s}\|ds.
\end{eqnarray}

For $J_{3}(t)$ and $J_{4}(t)$, by the Lipschitz continuity of $\bar{f}$, we obtain
\begin{eqnarray}  \label{J3}
\sup_{0\leq t\leq T}\|J_{3}(t)\|
\leq\!\!\!\!\!\!\!\!&& C
\int_{0}^{T} \| X_{s}^{\varepsilon}-\hat{X}^{\vare}_{s}\|ds
\end{eqnarray}
and
\begin{eqnarray}  \label{J4}
\sup_{0\leq t\leq T}\|J_{4}(t)\|\leq
C\int_{0}^{T}\|\hat{X}^{\vare}_{s}-\bar{X}_{s}\|ds.
\end{eqnarray}
Then by \eref{F3.19} to \eref{J4}, we have
\begin{eqnarray*}
\sup_{0\leq t\leq T}\|\hat{X}_{t}^{\vare}-\bar{X}_{t}\|
\leq\!\!\!\!\!\!\!\!&&C\int_{0}^{T}\big(1+\|\bar X_s\|^2_{1}+\|\hat X^{\vare}_{s}\|^2_{1}\big)\|\hat X^{\vare}_s-\bar X_{s}\|ds\nonumber\\
\!\!\!\!\!\!\!\!&&+ \sup_{0\leq t\leq T}\|J_{2}(t)\|+C\int_{0}^{T} \| X_{s}^{\varepsilon}-\hat{X}^{\vare}_{s}\|ds.
\end{eqnarray*}
The Gronwall's inequality and the definition of $\hat {\tau}^{\varepsilon}_{R}$ imply
\begin{eqnarray*}
\sup_{0\leq t\leq T\wedge \hat {\tau}^{\varepsilon}_{R}}\|\hat{X}_{t}^{\vare}-\bar{X}_{t}\|
\leq \!\!\!\!\!\!\!\!&&\left[\sup_{0\leq t\leq T\wedge \hat {\tau}^{\varepsilon}_{R}}\|J_{2}(t)\|+C\int_{0}^{T\wedge \hat {\tau}^{\varepsilon}_{R}} \| X_{s}^{\varepsilon}-\hat{X}_{s}\|ds\right]e^{\int_{0}^{T\wedge \hat {\tau}^{\varepsilon}_{R}}\big(1+\|\bar X_s\|^2_{1}+\|\hat X^{\vare}_{s}\|^2_{1}\big)ds}\\
\leq\!\!\!\!\!\!\!\!&& C_{T}e^{C_T R^2}\left[\sup_{0\leq t\leq T\wedge \hat {\tau}^{\varepsilon}_{R}}\|J_{2}(t)\|+\int_{0}^{T\wedge \hat {\tau}^{\varepsilon}_{R}} \| X_{s}^{\varepsilon}-\hat{X}_{s}\|ds\right].
\end{eqnarray*}
Note that $\hat {\tau}^{\varepsilon}_{R}\leq \tilde{\tau}^{\varepsilon}_{R}$ and by Lemma \ref{DEX}, we obtain
\begin{eqnarray}
\EE\left(\sup_{0\leq t\leq T\wedge \hat {\tau}^{\varepsilon}_{R}}\|\hat{X}_{t}^{\vare}-\bar{X}_{t}\|^p\right)
\leq\!\!\!\!\!\!\!\!&&C_{p,T}e^{C_{p,T}R^2}\EE\left[\sup_{0\leq t\leq T\wedge \hat {\tau}^{\varepsilon}_{R}}\|J_{2}(t)\|^p\right]\nonumber\\
\!\!\!\!\!\!\!\!&&+C_{p,T}e^{C_{p,T}R^6}(\|x\|^p_{\theta}+1)\delta^{\frac{p\theta}{2}}.\label{F3.23}
\end{eqnarray}

Now, it is remain to estimate $J_{2}(t)$. Set $n_{t}=[\frac{t}{\delta}]$, we write
\begin{align*}
J_{2}(t)=J_{2,1}(t)+J_{2,2}(t)+J_{2,3}(t),
\end{align*}
where
\begin{align*}
J_{2,1}(t)=\sum_{k=0}^{n_{t}-1}
\int_{k\delta}^{(k+1)\delta}e^{(t-s)A}\left[f(X_{k\delta}^{\varepsilon},\hat{Y}_{s}^{\varepsilon})-\bar{f}(X_{k\delta}^{\varepsilon})\right]ds,
\end{align*}
\begin{align*}
J_{2,2}(t)=\sum_{k=0}^{n_{t}-1}
\int_{k\delta}^{(k+1)\delta}e^{(t-s)A}\left[\bar{f}(X_{k\delta}^{\varepsilon})-\bar{f}(X_{s}^{\varepsilon})\right]ds,
\end{align*}
\begin{align*}
J_{2,3}(t)=
\int_{n_{t}\delta}^{t}e^{(t-s)A}\left[f(X_{n_{t}\delta}^{\varepsilon},\hat{Y}_{s}^{\varepsilon})-\bar{f}(X_{s}^{\varepsilon})\right]ds.
\end{align*}

For $J_{2,2}(t)$, we have
\begin{align}   \label{J42}
\sup_{0\leq t\leq T\wedge \tilde{\tau}^{\varepsilon}_{R}} \|J_{2,2}(t) \|
\leq C\int_{0}^{T\wedge \tilde{\tau}^{\varepsilon}_{R}}\|X_{s(\delta)}^{\varepsilon}-X_{s}^{\varepsilon} \|ds.
\end{align}

For $J_{2,3}(t)$, it follows from the boundness of $f$,
\begin{eqnarray}  \label{J43}
\sup_{0\leq t\leq T\wedge \tilde{\tau}^{\varepsilon}_{R}} \| J_{2,3}(t) \|\leq\delta.
\end{eqnarray}

For $J_{2,1}(t)$,
from the construction of $\hat{Y}_{t}^{\varepsilon}$,
we obtain that, for any $k\in \mathbb{N}_{\ast}$ and $s\in[0,\delta)$,
\begin{eqnarray}
\hat{Y}_{s+k\delta}^{\varepsilon}
=\!\!\!\!\!\!\!\!&&\hat Y_{k\delta}^{\varepsilon}+\frac{1}{\varepsilon}\int_{k\delta}^{k\delta+s}A\hat{Y}_{r}^{\varepsilon}dr
+\frac{1}{\varepsilon}\int_{k\delta}^{k\delta+s}g(X_{k\delta}^{\varepsilon},\hat{Y}_{r}^{\varepsilon})dr
+\frac{1}{\varepsilon^{1/\alpha}}\int_{k\delta}^{k\delta+s}dZ_r   \nonumber\\
=\!\!\!\!\!\!\!\!&&\hat Y_{k\delta}^{\varepsilon}+\frac{1}{\varepsilon}\int_{0}^{s}A\hat{Y}_{r+k\delta}^{\varepsilon}dr
+\frac{1}{\varepsilon}\int_{0}^{s}g(X_{k\delta}^{\varepsilon},\hat{Y}_{r+k\delta}^{\varepsilon})dr
+\frac{1}{\varepsilon^{1/\alpha}}\int_{0}^{s}d Z_{k\delta}(r) ,\label{E3.26}
\end{eqnarray}
where $Z_{k\delta}(t):=Z_{t+k\delta}-Z_{k\delta}$ is the shift version of $Z_t$, which is also a cylindrical $\alpha$-stable process.

Recall that $\bar{Z}_t$ be a cylindrical $\alpha$-stable process
which is independent of $(X_{k\delta}^{\varepsilon},\hat Y_{k\delta}^{\varepsilon})$.
We construct a process $Y^{X_{k\delta}^{\varepsilon},\hat Y_{k\delta}^{\varepsilon}}_t$ by means of $Y^{x,y}_t \mid_{(x,y)=(X_{k\delta}^{\varepsilon},\hat Y_{k\delta}^{\varepsilon})}$, i.e.,
\begin{eqnarray}
Y_{\frac{s}{\varepsilon}}^{X_{k\delta}^{\varepsilon},\hat Y_{k\delta}^{\varepsilon}}=\!\!\!\!\!\!\!\!&&\hat Y_{k\delta}^{\varepsilon}
+\int_{0}^{\frac{s}{\varepsilon}}AY_{r}^{X_{k\delta}^{\varepsilon},\hat Y_{k\delta}^{\varepsilon}}dr
+\int_{0}^{\frac{s}{\varepsilon}}g(X_{k\delta}^{\varepsilon}, Y_{r}^{X_{k\delta}^{\varepsilon},\hat Y_{k\delta}^{\varepsilon}})dr
+\int_{0}^{\frac{s}{\varepsilon}}d\bar{Z}_r  \nonumber\\
=\!\!\!\!\!\!\!\!&&\hat Y_{k\delta}^{\varepsilon}
+\frac{1}{\varepsilon}\int_{0}^{s}AY_{\frac{r}{\varepsilon}}^{X_{k\delta}^{\varepsilon},\hat Y_{k\delta}^{\varepsilon}}dr
+\frac{1}{\varepsilon}\int_{0}^{s}g(X_{k\delta}^{\varepsilon},Y_{\frac{r}{\varepsilon}}^{X_{k\delta}^{\varepsilon},\hat Y_{k\delta}^{\varepsilon}})dr
+\frac{1}{\varepsilon^{1/ \alpha}}\int_{0}^{s}d\hat Z_r,  \label{E3.27}
\end{eqnarray}
where $\hat Z_t:=\vare^{1/ \alpha}\bar{Z}_{\frac{t}{\vare}}$ is again a cylindrical $\alpha$-stable process by self-similar property of stable L\'evy processes. Then the uniqueness of the solution to Eq. \eref{E3.26} and Eq. \eref{E3.27} implies
that the distribution of $(X_{k\delta}^{\varepsilon},\hat{Y}^{\vare}_{s+k\delta})_{0\leq s\leq \delta}$
coincides with the distribution of
$(X_{k\delta}^{\varepsilon},
Y_{\frac{s}{\varepsilon}}^{X_{k\delta}^{\varepsilon},\hat Y_{k\delta}^{\varepsilon}})_{0\leq s\leq \delta}$.

Then we try to control  $\|J_{2,1}(t)\|$:
\begin{eqnarray}
\!\!\!\!\!\!\!\!&& \EE\left(\sup_{0\leq t\leq T} \| J_{2,1}(t) \|^{2}\right) \nonumber\\
=\!\!\!\!\!\!\!\!&&
\EE\sup_{0\leq t\leq T}\Big \|\sum_{k=0}^{n_{t}-1}e^{(t-(k+1)\delta)A}
\int_{k\delta}^{(k+1)\delta}e^{((k+1)\delta-s)A}
\left[f(X_{k\delta}^{\varepsilon},\hat{Y}_{s}^{\varepsilon})
- \bar{f}(X_{k\delta}^{\varepsilon})\right]ds\Big \|^{2}
\nonumber\\
\leq\!\!\!\!\!\!\!\!&& \EE\sup_{0\leq t\leq T}
\left\{n_{t}\sum_{k=0}^{n_{t}-1}\Big \|\int_{k\delta}^{(k+1)\delta}
e^{((k+1)\delta-s)A}\left[f(X_{k\delta}^{\varepsilon},\hat{Y}_{s}^{\varepsilon})-\bar{f}(X_{k\delta}^{\varepsilon})\right]ds\Big \|^{2}\right\}
\nonumber\\
\leq\!\!\!\!\!\!\!\!&&
[\frac{T}{\delta}]
\sum_{k=0}^{[\frac{T}{\delta}]-1}
\mathbb{E} \Big \|\int_{k\delta}^{(k+1)\delta}
e^{((k+1)\delta-s)A}\left[f(X_{k\delta}^{\varepsilon},\hat{Y}_{s}^{\varepsilon})-\bar{f}(X_{k\delta}^{\varepsilon})\right]ds\Big\|^{2}
\nonumber\\
\leq\!\!\!\!\!\!\!\!&&
\frac{C_{T}}{\delta^{2}}\max_{0\leq k\leq[\frac{T}{\delta}]-1}\mathbb{E}
\Big \| \int_{k\delta}^{(k+1)\delta}
e^{((k+1)\delta-s)A}\left[f(X_{k\delta}^{\varepsilon},\hat{Y}_{s}^{\varepsilon})-\bar{f}(X_{k\delta}^{\varepsilon})\right]ds \Big \|^{2}  \nonumber\\
=\!\!\!\!\!\!\!\!&&
C_{T}\frac{\vare^{2}}{\delta^{2}}\max_{0\leq k\leq[\frac{T}{\delta}]-1}
\mathbb{E}
\Big\| \int_{0}^{\frac{\delta}{\varepsilon}}
e^{(\delta-s\varepsilon)A}
\left[f(X_{k\delta}^{\varepsilon},\hat{Y}_{s\varepsilon+k\delta}^{\varepsilon})-\bar{f}(X_{k\delta}^{\varepsilon})\right]ds\Big\|^{2}  \nonumber\\
=\!\!\!\!\!\!\!\!&&C_{T}\frac{\varepsilon^{2}}{\delta^{2}}\max_{0\leq k\leq[\frac{T}{\delta}]-1}\int_{0}^{\frac{\delta}{\varepsilon}}
\int_{r}^{\frac{\delta}{\varepsilon}}\Psi_{k}(s,r)dsdr,  \nonumber
\end{eqnarray}
where
\begin{eqnarray}
\Psi_{k}(s,r)=\!\!\!\!\!\!\!\!&&\mathbb{E}\left\langle e^{(\delta-s\varepsilon)A}
\big(f(X_{k\delta}^{\varepsilon},\hat{Y}_{s\varepsilon+k\delta}^{\varepsilon})-\bar{f}(X_{k\delta}^{\varepsilon})\big), e^{(\delta-r\varepsilon)A}
\big(f(X_{k\delta}^{\varepsilon},\hat{Y}_{r\varepsilon+k\delta}^{\varepsilon})-\bar{f}(X_{k\delta}^{\varepsilon})\big)\right\rangle  \nonumber\\
=\!\!\!\!\!\!\!\!&&\mathbb{E}\left\langle e^{(\delta-s\varepsilon)A}
\big(f(X_{k\delta}^{\varepsilon},Y_{s}^{X_{k\delta}^{\varepsilon},\hat Y_{k\delta}^{\varepsilon}})-\bar{f}(X_{k\delta}^{\varepsilon})\big), e^{(\delta-r\varepsilon)A}\big(f(X_{k\delta}^{\varepsilon}, Y_{r}^{X_{k\delta}^{\varepsilon},\hat Y_{k\delta}^{\varepsilon}})-\bar{f}(X_{k\delta}^{\varepsilon})\big)\right\rangle.  \nonumber
\end{eqnarray}
Now, let's estimate $\Psi_{k}(s,r)$. Define
$$
\mathcal{F}_s:=\sigma\{ Y_{u}^{x,y},u\leq s\}.
$$
By the Markov property, Proposition \ref{ergodicity} and condition \ref{A3}, we have for $s>r$,
\begin{eqnarray*}
\Psi_{k}(s,r)=\!\!\!\!\!\!\!\!&&\mathbb{E}\left[\EE\left\langle e^{(\delta-s\varepsilon)A}
\big(f(x,Y_{s}^{x,y})-\bar{f}(x)\big), e^{(\delta-r\varepsilon)A}\big(f(x, Y_{r}^{x,y})-\bar{f}(x)\big)\right\rangle \mid_{(x,y)=(X_{k\delta}^{\varepsilon},\hat{Y}^{\vare}_{s+k\delta})}\right]  \nonumber\\
=\!\!\!\!\!\!\!\!&&\mathbb{E}\left\{\EE\left[\left\langle e^{(\delta-s\varepsilon)A}
\EE \big[f(x,Y_{s}^{x,y})-\bar{f}(x)\mid \mathcal{F}_{r}\big],\right.\right.\right.\\
 &&\quad\quad\quad \left.\left.\left.e^{(\delta-r\varepsilon)A}(f(x, Y_{r}^{x,y})-\bar{f}(x)\big)\right\rangle\right]\mid_{(x,y)=(X_{k\delta}^{\varepsilon},\hat{Y}^{\vare}_{s+k\delta})}\right\}\nonumber\\
\leq\!\!\!\!\!\!\!\!&&C\mathbb{E}\left\{\EE\left[\|\EE f(x,Y_{s-r}^{x,z})-\bar{f}(x)\|\mid_{z=Y_{r}^{x,y}}\right]\mid_{(x,y)=(X_{k\delta}^{\varepsilon},\hat Y_{k\delta}^{\varepsilon})}\right\}\nonumber\\
\leq\!\!\!\!\!\!\!\!&&C\mathbb{E}\left[\EE(1+\|x\|+\|Y^{x,y}_{r}\|)e^{-(\lambda_1-L_g)(s-r)}\mid_{(x,y)=(X_{k\delta}^{\varepsilon},\hat Y_{k\delta}^{\varepsilon})}\right]\nonumber\\
=\!\!\!\!\!\!\!\!&&C\mathbb{E}\left[(1+\|X_{k\delta}^{\varepsilon}\|+\|Y^{X_{k\delta}^{\varepsilon},\hat Y_{k\delta}^{\varepsilon}}_{r}\|)e^{-(\lambda_1-L_g)(s-r)}\right]\nonumber\\
=\!\!\!\!\!\!\!\!&&C\mathbb{E}\left(1+\|X_{k\delta}^{\varepsilon}\|+\|\hat Y^{\vare}_{k\delta+\varepsilon r}\|\right)e^{-(\lambda_1-L_g)(s-r)}\nonumber\\
\leq\!\!\!\!\!\!\!\!&&C_T(1+ \|x \|+ \|y\|)e^{-(\lambda_1-L_g)(s-r)},
\end{eqnarray*}
where the last inequality comes from Lemmas \ref{PMY} and \ref{MDY}. As a consequence,
\begin{eqnarray}  \label{J412}
\mathbb{E}\left(\sup_{0\leq t\leq T} \| J_{2,1}(t) \|^{2}\right)
\leq\!\!\!\!\!\!\!\!&&
C_{T}\frac{\varepsilon^{2}}{\delta^{2}}(1 + \| x \| + \| y \|)
\int_{0}^{\frac{\delta}{\varepsilon}}\int_{r}^{\frac{\delta}{\varepsilon}}e^{-(\lambda_1-L_g)(s-r)}dsdr    \nonumber\\
\leq\!\!\!\!\!\!\!\!&&C_{T}\frac{\varepsilon}{\delta}(1+ \| x \|+ \| y \|).
\end{eqnarray}
This, together with \eref{J42}, \eref{J43}, \eref{J412} and Lemma \ref{COX}, we get
\begin{eqnarray}
\mathbb{E}\left(\sup_{0\leq t\leq T\wedge \hat{\tau}^{\varepsilon}_{R}} \| J_{2}(t) \|^{p}\right)
\leq\!\!\!\!\!\!\!\!&&C_{p,T}e^{C_{p,T}R^2}\Big[(\|x\|^p_{\theta}+1)\delta^{\frac{p\theta}{2}}+(1+\|x\|+\|y\|)^{p/2}\frac{\vare^{p/2}}{\delta^{p/2}}\Big].  \label{J2}
\end{eqnarray}
According to the estimates (\ref{F3.23}) and (\ref{J2}), we obtain
\begin{eqnarray*}
\mathbb{E}\Big(\sup_{0\leq t\leq  T\wedge \hat{\tau}^{\varepsilon}_{R}}
\| \hat{X}_{t}^{\vare}-\bar{X}_{t}\|^{p}\Big)
\leq \!\!\!\!\!\!\!\!&&
C_{p,T}e^{C_{p,T}R^6}\Big[(\|x\|^p_{\theta}+1)\delta^{\frac{p\theta}{2}}+\delta^p+(1+\|x\|+\|y\|)^{p/2}\frac{\vare^{p/2}}{\delta^{p/2}}\Big].
\end{eqnarray*}
The proof is complete.
\end{proof}

\subsection{Proof of Theorem \ref{main result 1}} \label{sub 3.4}

\begin{proof}
 By Lemmas \ref{PMY}, \ref{DEX}, \ref{barX} and \ref{ESX}, we have for any $1\leq p<p'<\alpha$,
\begin{eqnarray}
&&\mathbb{E}\left(\sup_{0\leq t\leq T} \| X_{t}^{\vare}-\bar{X}_{t}\|^{p}\right)\nonumber\\
\leq\!\!\!\!\!\!\!\!&&
\mathbb{E}\left(\sup_{0\leq t\leq T\wedge \hat{\tau}^{\varepsilon}_{R}} \| X_{t}^{\vare}-\bar{X}_{t} \|^p\right)
+
\mathbb{E}\left(\sup_{0\leq t\leq T} \| X_{t}^{\vare}-\bar{X}_{t} \|^p 1_{\{T> \hat{\tau}^{\vare}_R\}}\right)\nonumber\\
\leq\!\!\!\!\!\!\!\!&&
\mathbb{E}\left(\sup_{0\leq t\leq T\wedge \hat{\tau}^{\varepsilon}_{R}} \|X_{t}^{\vare}-\hat {X}^{\vare}_{t}\|^p\right)
+
\mathbb{E}\left(\sup_{0\leq t\leq T\wedge \hat{\tau}^{\varepsilon}_{R}} \|\hat{X}_{t}^{\vare}-\bar{X}_{t}\|^p\right)\nonumber\\
\!\!\!\!\!\!\!\!&&
+
C\left[\mathbb{E}\sup_{0\leq t\leq T} \left(\| X_{t}^{\vare}\|^{p'}+\|\bar{X}_{t} \|^{p'}\right)\right]^{p/p'}\left[\mathbb{P}(T> \hat{\tau}^{\varepsilon}_{R})\right]^{(p'-p)/p'}\nonumber\\
\leq\!\!\!\!\!\!\!\!&&
C_{p,T}e^{C_{p,T}R^6}\Big[(\|x\|_{\theta}+1)^p \delta^{\frac{p\theta}{2}}
+\delta^p+(1+\|x\|+\|y\|)^{p/2}\frac{\vare^{p/2}}{\delta^{p/2}}\Big]+\frac{C_{p,p',T}(1+\|x\|)}{R^{(p'-p)/p'}}, \label{F3.29}
\end{eqnarray}
where the last inequality comes from the chebyshev's inequality. Indeed, by Lemmas \ref{PMY}, \ref{MDY} and \ref{barX}, we get
\begin{eqnarray*}
\mathbb{P}(T> \hat{\tau}^{\varepsilon}_{R})\leq\!\!\!\!\!\!\!\!&&\EE\left[\sup_{0\leq t\leq T}\|X^{\vare}_t\|+\int^T_0\frac{\|X^{\vare}_s\|^2_1}{(\|X^{\vare}_s\|^2+1)^{1/2}}ds\right]/R\\
+\!\!\!\!\!\!\!\!&&\EE\left[\sup_{0\leq t\leq T}\|\hat X^{\vare}_t\|+\int^T_0\frac{\|\hat X^{\vare}_s\|^2_1}{(\|\hat X^{\vare}_s\|^2+1)^{1/2}}ds\right]/R\\
+\!\!\!\!\!\!\!\!&&\EE\left[\sup_{0\leq t\leq T}\|\bar X_t\|+\int^T_0\frac{\|\bar X_s\|^2_1}{(\|\bar X_s\|^2+1)^{1/2}}ds\right]/R\\
\leq\!\!\!\!\!\!\!\!&&\frac{C_T(1+\|x\|)}{R}.
\end{eqnarray*}

Now, taking $\delta=\vare^{\frac{1}{\theta+1}}$ and $R=\left(\frac{\theta p}{4(\theta+1)C_{p,T}}\ln{1/\vare}\right)^{1/6}$, we have
\begin{eqnarray*}
\mathbb{E}\left(\sup_{0\leq t\leq T}\|X_{t}^{\vare}-\bar{X}_{t}\|^p\right)\leq \frac{C_{p,p',T}(1+\|x\|^p_{\theta}+\|y\|^{p/2})}{\left(-\ln \vare\right)^{\frac{p'-p}{6p'}}},
\end{eqnarray*}
which implies for any $k<\frac{\alpha-p}{6\alpha}$,
\begin{eqnarray*}
\mathbb{E}\left(\sup_{0\leq t\leq T}\|X_{t}^{\vare}-\bar{X}_{t}\|^p\right)\leq\frac{C_{p,k,T}(1+\|x\|^p_{\theta}+\|y\|^{p/2})}{\left(-\ln \vare\right)^{k}}.\label{L1}
\end{eqnarray*}
The proof is complete.
\end{proof}

\section{Appendix} \label{Sec appendix}

\subsection{Some estimates about the nonlinearity $N$}
The proof of \eref{P3}-\eref{P5} can be founded in \cite[Appendix]{Xu1} and we show \eref{P2} here.

By the Sobolev embedding theorem, for any $\sigma\in [0, 1/2)$,
\begin{eqnarray*}
|\langle N(x), y\rangle|=\!\!\!\!\!\!\!\!&&|\int_{\mathbb{T}}x(\xi)y(\xi)d\xi-\int_{\mathbb{T}}x^3(\xi)y(\xi)d\xi|\\
\leq\!\!\!\!\!\!\!\!&&\|x\|\|y\|+\left[\int_{\mathbb{T}} |x(\xi)|^{\frac{6}{2\sigma+1}}d\xi\right]^{\frac{2\sigma+1}{2}}\left[\int_{\mathbb{T}} |y(\xi)|^{\frac{2}{1-2\sigma}}d\xi\right]^{\frac{1-2\sigma}{2}}\\
=\!\!\!\!\!\!\!\!&&\|x\|\|y\|+\|x\|^{3}_{L^{\frac{6}{2\sigma+1}}(\mathbb{T})}\|y\|_{L^{\frac{2}{1-2\sigma}}(\mathbb{T})}\\
\leq\!\!\!\!\!\!\!\!&&\|x\|\|y\|+C\|x\|^{3}_{\frac{1-\sigma}{3}}\|y\|_{\sigma}\\
\leq\!\!\!\!\!\!\!\!&&C\left(1+\|x\|^3_{\frac{1-\sigma}{3}}\right)\|y\|_{\sigma},
\end{eqnarray*}
which implies
$$
\|N(x)\|_{-\sigma}\leq C\left(1+\|x\|^3_{\frac{1-\sigma}{3}}\right).
$$
\subsection{The existence and uniqueness of solution of system \eref{main equation}} Fix $\varepsilon>0$, for all $\omega\in\Omega$, define
$$
W^{\varepsilon}_t(\omega):=X^{\vare}_{t}(\omega)-L_A(t,\omega),\quad V^{\varepsilon}_t(\omega):=Y^{\vare}_{t}(\omega)-Z^{\varepsilon}_A(t,\omega),
$$
where $L_A(t)=\int^t_0 e^{(t-s)A}d L_s$ and $Z^{\varepsilon}_A(t):=\frac{1}{\vare^{1/\alpha}}\int^t_0 e^{(t-s)A/\vare}d Z_s$. Then
\begin{equation}\left\{\begin{array}{l}\label{A4.1}
\displaystyle
\partial_t W^{\varepsilon}_t= A W^{\varepsilon}_t+N(W^{\varepsilon}_t+L_A(t))+f(W^{\varepsilon}_t+L_A(t),V^{\varepsilon}_t+Z^{\varepsilon}_A(t)),\quad W^{\varepsilon}_0=x \\
\partial_t V^{\varepsilon}_t= \frac{1}{\varepsilon}\left[A V^{\varepsilon}_t+g(W^{\varepsilon}_t+L_A(t),V^{\varepsilon}_t+Z^{\varepsilon}_A(t))\right],\quad V^{\varepsilon}_0=y.\end{array}\right.
\end{equation}
For each $T>0$, define
$$
K^{\vare}_T(\omega):=\sup_{t\leq T}\|L_A(t,\omega)\|_{1}+\int^T_0\|Z^{\vare}_{A}(t,\omega)\|dt.
$$
By Remark \ref{Re2}, for every $k\in \mathbb{N}$, there exists some set $N^{\vare}_k$ such that $\PP(N^{\vare}_k)=0$ and
$$
K^{\vare}_k(\omega)<\infty,\quad \omega\not\in N^{\vare}_k.
$$
Define $N^{\vare}=\cup_{k\geq 1}N^{\vare}_k$, it is easy to see $\PP(N^{\vare})=0$ and that for all $T>0$
$$
K^{\vare}_T(\omega)<\infty,\quad \omega\not\in N^{\vare}.
$$

\begin{lemma} \label{AL4.1} Assume the conditions \ref{A1}-\ref{A3} hold. Then the following statements hold.\\
(i) For every $\varepsilon>0$, $x\in H$, $y\in H$ and $\omega\not\in N^{\vare}$, there exists some $0<T(\omega)<1$, depending on $\|x\|$ and $K^{\vare}_1(\omega)$, such that system \eref{A4.1} admits a unique solution $W^{\vare}.(\omega)\in C([0,T];H)$ and $V^{\vare}.(\omega)\in C([0,T];H)$ satisfying for all $\sigma\in[1/6,1/2]$,
\begin{eqnarray*}
\|W^{\vare}_t\|_{2\sigma}\leq C(t^{-\sigma}+1),
\end{eqnarray*}
where $C$ is some constant depending on $\|x\|$, $\sigma$ and $K^{\vare}_1(\omega)$.\\
(ii) Let $\sigma\in [1/6, 1/2]$. For every $\varepsilon>0$, $x\in H^{2\sigma}$, $y\in H$ and $\omega\not\in N^{\vare}$, there exists some $0<\tilde T(\omega)<1$, depending on $\|x\|$, $\sigma$ and $K^{\vare}_1(\omega)$, such that system \eref{A4.1} admits a unique solution $W^{\vare}.(\omega)\in C([0,\tilde T]; H^{2\sigma})$ and $V^{\vare}.(\omega)\in C([0,\tilde T];H)$ satisfying
$$
\sup_{0\leq t\leq \tilde T}\|W^{\vare}_t\|_{2\sigma}\leq 1+\|x\|_{2\sigma}.
$$
\end{lemma}
\begin{proof}
We shall apply the Banach fixed point theorem. Since these two statements will be proved by the same method, we only prove statement $(i)$.

Let $0<T\leq 1$ and $B>0$ be some
constants to be determined later. For $\sigma=\frac{1}{6}$, define
\begin{eqnarray*}
S:=\!\!\!\!\!\!\!\!&&\Big\{u=(u_1,u_2): u_i\in C([0, T]; H),i=1,2, u_1(0)=x,u_2(0)=y, \\
&&\quad u_1(t)\in H^{2\sigma},\forall t\in (0,T], \sup_{0\leq t\leq T}\left[t^{\sigma}\|u_1(t)\|_{2\sigma}\right]+\sup_{0\leq t\leq T}\|u_1(t)\|+\sup_{0\leq t\leq T}\|u_2(t)\|\leq B\Big\}.
\end{eqnarray*}
Given any $u=(u_1,u_2)$, $v=(v_1, v_2)\in S$, define
$$
d(u,v)=\sup_{0\leq t\leq T}\left[t^{\sigma}\|u_1(t)-v_1(t)\|_{2\sigma}\right]+\sup_{0\leq t\leq T}\|u_1(t)-v_1(t)\|+\sup_{0\leq t\leq T}\|u_2(t)-v_2(t)\|.
$$
Then $(S, d)$ is a closed metric space. Further define a map $\mathcal{F}: S\rightarrow (C([0, T]; H),C([0, T]; H))$ as the
following: for any $u=(u_1,u_2)\in S$,
$$\mathcal{F}(u)(t)=\left(\mathcal{F}(u)_1(t),\mathcal{F}(u)_2(t)\right),$$
where
\begin{equation}\left\{\begin{array}{l}\label{A4.2}
\displaystyle
\mathcal{F}(u)_1(t)\!=\!e^{tA}x+\!\!\int^t_0\!\! e^{(t-s)A}N(u_1(s)\!+\!L_A(s))ds+\!\!\int^t_0\!\!e^{(t-s)A}\!\!f(u_1(s)+L_A(s),u_2(s)+Z^{\varepsilon}_A(s))ds, \\
\mathcal{F}(u)_2(t)\!=\!e^{tA/\varepsilon}y+\frac{1}{\varepsilon}\int^t_0e^{(t-s)A/\vare}g(u_1(s)+L_A(s),u_2(s)+Z^{\varepsilon}_A(s))ds.\end{array}\right.
\end{equation}

We shall prove that there exist $T_0>0$ and $B_0>0$ such that whenever $ T\in(0, T_0]$ and $B>B_0$, the following two statements hold:\\
$\text{(a)}\quad \mathcal{F}(u)\in S \quad \text{for} \quad u\in S;$\\
$\text{(b)}\quad d(\mathcal{F}(u),\mathcal{F}(v))\leq \frac{1}{2}d(u,v) \quad \text{for} \quad u,v\in S.$\\
\vspace{0.5cm}
It is obvious that $\mathcal{F}(u)(0)=(x,y)$. By the condition \ref{A3}, properties \eref{P1} and \eref{P2},
we have
\begin{eqnarray*}
\|\mathcal{F}(u)_1(t)\|_{2\sigma}\leq\!\!\!\!\!\!\!\!&& Ct^{-\sigma}\|x\|+C\int^t_0(t-s)^{-\sigma}(1+\|u_1(s)\|^3_{2\sigma}+\|L_A(s)\|^3_{2\sigma})ds+C\int^t_0 (t-s)^{-\sigma}ds\\
\leq\!\!\!\!\!\!\!\!&&Ct^{-\sigma}\|x\|+C\int^t_0(t-s)^{-\sigma}\left[1+(K^{\vare}_1)^3+\|u_1(s)\|^3_{2\sigma}\right]ds+Ct^{1-\sigma},
\end{eqnarray*}
which implies
\begin{eqnarray}
\sup_{0\leq t\leq T}\left[t^{\sigma}\|\mathcal{F}(u)_1(t)\|_{2\sigma}\right]\leq\!\!\!\!\!\!\!\!&& C\|x\|+C\sup_{0\leq t\leq T}\left[t^{\sigma}\int^t_0(t-s)^{-\sigma}(1+(K^{\vare}_1)^3+s^{-3\sigma}B^3)ds\right]\nonumber\\
\leq\!\!\!\!\!\!\!\!&& C\|x\|+CT^{1-3\sigma}\left[1+(K^{\vare}_1)^3+B^3\right].\label{A4.3}
\end{eqnarray}
Meanwhile,
\begin{eqnarray}
\sup_{0\leq t\leq T}\|\mathcal{F}(u)_1(t)\|\leq\!\!\!\!\!\!\!\!&&\|x\|+ \int^T_0 \|N(u_1(s)+L_A(s))\|ds\nonumber\\
\!\!\!\!\!\!\!\!&&+\int^T_0\|f(u_1(s)+L_A(s),u_2(s)+Z^{\varepsilon}_A(s))\|ds\nonumber\\
\leq\!\!\!\!\!\!\!\!&&\|x\|+C\int^T_0(1+\|u_1(s)\|^3_{2\sigma}+\|L_A(s)\|^3_{2\sigma})ds\nonumber\\
\leq\!\!\!\!\!\!\!\!&&\|x\|+CT^{1-3\sigma}\left[1+(K^{\vare}_1)^3+B^3\right].\label{A4.4}
\end{eqnarray}
Furthermore,
\begin{eqnarray}
\sup_{0\leq t\leq T}\|\mathcal{F}(u)_2(t)\|\leq\!\!\!\!\!\!\!\!&& \|y\|+\frac{1}{\varepsilon}\left[\int^T_0\|g(u_1(s)+L_A(s),u_2(s)+Z^{\varepsilon}_A(s))\|ds\right]\nonumber\\
\leq\!\!\!\!\!\!\!\!&&\|y\|+\frac{CT}{\varepsilon}\left[\sup_{0\leq s\leq T}\left(1+\|u_1(s)\|+\|u_2(s)\|+\|L_A(s)\|\right)\right]+\frac{C}{\varepsilon}\int^T_0\|Z^{\varepsilon}_A(s)\|ds\nonumber\\
\leq\!\!\!\!\!\!\!\!&&\|y\|+\frac{CT}{\varepsilon}\left(1+B+K^{\varepsilon}_1\right)+\frac{C}{\varepsilon}K^{\varepsilon}_1.\label{A4.5}
\end{eqnarray}
It is easy to see the continuity of $\mathcal{F}(u)_1$ and $\mathcal{F}(u)_2$. As $T>0$ is sufficiently small and $B$ is large enough, statement $(a)$ follows from  \eref{A4.3}-\eref{A4.5}.

Now, let's prove statement $(b)$. Given any $u=(u_1,u_2)$, $v=(v_1, v_2)\in S$, by \eref{P4}
\begin{eqnarray*}
&&t^{\sigma}\|\mathcal{F}(u)_1(t)-\mathcal{F}(v)_1(t)\|_{2\sigma}\\
\leq\!\!\!\!\!\!\!\!&& Ct^{\sigma}\int^t_0 (t-s)^{-\sigma}\|N(u_1(s)+L_A(s))-N(v_1(s)+L_A(s))\|ds\\
\!\!\!\!\!\!\!\!&&+Ct^{\sigma}\int^t_0(t-s)^{\sigma}\|f(u_1(s)+L_A(s),u_2(s)+Z^{\varepsilon}_A(s))-f(v_1(s)+L_A(s),v_2(s)+Z^{\varepsilon}_A(s))\|ds\\
\leq\!\!\!\!\!\!\!\!&& Ct^{\sigma}\int^t_0 (t-s)^{-\sigma}\left[1+(K^{\vare}_1)^2+\|u_1(s)\|^2_{2\sigma}+\|v_1(s)\|^2_{2\sigma}\right]\|u_1(s)-v_1(s)\|_{2\sigma}ds\\
\!\!\!\!\!\!\!\!&&+Ct^{\sigma}\int^t_0(t-s)^{\sigma}\left[\|u_1(s)-v_1(s)\|+\|u_2(s)-v_2(s)\|\right]ds
\end{eqnarray*}
Note that $\|u_1(s)\|_{2\sigma}\leq s^{-\sigma}B$ and $\|v_1(s)\|_{2\sigma}\leq s^{-\sigma}B$, we have
\begin{eqnarray*}
t^{\sigma}\|\mathcal{F}(u)_1(t)-\mathcal{F}(v)_1(t)\|_{2\sigma}\leq\!\!\!\!\!\!\!\!&& C\left[1+(K^{\vare}_1)^2\right]t^{\sigma}\int^t_0 (t-s)^{-\sigma}s^{-\sigma}\left(s^{\sigma}\|u_1(s)-v_1(s)\|_{2\sigma}\right)ds\\
\!\!\!\!\!\!\!\!&&+CB^2t^{\sigma}\int^t_0 (t-s)^{-\sigma}s^{-3\sigma}s^{\sigma}\|u_1(s)-v_1(s)\|_{2\sigma}ds\\
\!\!\!\!\!\!\!\!&&+Ct^{\sigma}\int^t_0(t-s)^{-\sigma}\left[\|u_1(s)-v_1(s)\|+\|u_2(s)-v_2(s)\|\right]ds\\
\leq\!\!\!\!\!\!\!\!&& C\left[1+(K^{\vare}_1)^2\right]t^{1-\sigma}\sup_{0\leq s\leq t}\left[s^{\sigma}\|u_1(s)-v_1(s)\|_{2\sigma}\right]\\
\!\!\!\!\!\!\!\!&&+CB^2t^{1-3\sigma}\sup_{0\leq s\leq t}\left[s^{\sigma}\|u_1(s)-v_1(s)\|_{2\sigma}\right]\\
\!\!\!\!\!\!\!\!&&+Ct\sup_{0\leq s\leq t}\left[\|u_1(s)-v_1(s)\|+\|u_2(s)-v_2(s)\|\right]\\
\leq\!\!\!\!\!\!\!\!&& C\left[1+(K^{\vare}_1)^2+B^2\right]t^{1-3\sigma}d(u,v).
\end{eqnarray*}
This implies
\begin{eqnarray}
\sup_{0\leq t\leq T}\left[t^{\sigma}\|\mathcal{F}(u)_1(t)-\mathcal{F}(v)_1(t)\|_{2\sigma}\right]\leq\!\!\!\!\!\!\!\!&&C\left[1+(K^{\vare}_1)^2+B^2\right]T^{1-3\sigma}d(u,v).\label{A4.6}
\end{eqnarray}
Meanwhile, by \eref{P4}, we obtain
\begin{eqnarray*}
&&\|\mathcal{F}(u)_1(t)-\mathcal{F}(v)_1(t)\|\\
\leq\!\!\!\!\!\!\!\!&&\int^t_0 \|N(u_1(s)+L_A(s))-N(v_1(s)+L_A(s))\|ds\nonumber\\
\!\!\!\!\!\!\!\!&&+\int^t_0\|f(u_1(s)+L_A(s),u_2(s)+Z^{\varepsilon}_A(s))-f(v_1(s)+L_A(s),v_2(s)+Z^{\varepsilon}_A(s))\|ds\nonumber\\
\leq\!\!\!\!\!\!\!\!&&C\int^t_0 \left[1+(K^{\vare}_1)^2+\|u_1(s)\|^2_{2\sigma}+\|v_1(s)\|^2_{2\sigma}\right]\|u_1(s)-v_1(s)\|_{2\sigma}ds\nonumber\\
\!\!\!\!\!\!\!\!&&+C\int^t_0\left[\|u_1(s)-v_1(s)\|+\|u_2(s)-v_2(s)\|\right]ds\nonumber\\
\leq\!\!\!\!\!\!\!\!&&C\left[1+(K^{\vare}_1)^2\right]\int^t_0s^{-\sigma}\left[s^{\sigma}\|u_1(s)-v_1(s)\|_{2\sigma}\right]ds+CB^2\int^t_0s^{-3\sigma}\left[s^{\sigma}\|u_1(s)-v_1(s)\|_{2\sigma}\right]ds\nonumber\\
\!\!\!\!\!\!\!\!&&+C\int^t_0\left[\|u_1(s)-v_1(s)\|+\|u_2(s)-v_2(s)\|\right]ds,\nonumber
\end{eqnarray*}
which implies
\begin{eqnarray}
\sup_{0\leq t\leq T}\|\mathcal{F}(u)_1(t)-\mathcal{F}(v)_1(t)\|\leq\!\!\!\!\!\!\!\!&&CT^{1-\sigma}\left[1+(K^{\vare}_1)^2\right]\sup_{0\leq t\leq T}\left[t^{\sigma}\|u_1(t)-v_1(t)\|_{2\sigma}\right]\nonumber\\
\!\!\!\!\!\!\!\!&&+CB^2T^{1-3\sigma}\sup_{0\leq t\leq T}\left[t^{\sigma}\|u_1(t)-v_1(t)\|_{2\sigma}\right]\nonumber\\
\!\!\!\!\!\!\!\!&&+CT\sup_{0\leq t\leq T}\left[\|u_1(t)-v_1(t)\|+\|u_2(t)-v_2(t)\|\right],\nonumber\\
\leq\!\!\!\!\!\!\!\!&& C\left[1+(K^{\vare}_1)^2+B^2\right]T^{1-3\sigma}d(u,v).\label{A4.7}
\end{eqnarray}
Furthermore,
\begin{eqnarray}
&&\sup_{0\leq t\leq T}\|\mathcal{F}(u)_2(t)-\mathcal{F}(v)_2(t)\|\nonumber\\
\leq\!\!\!\!\!\!\!\!&& \frac{1}{\varepsilon}\int^T_0\|g(u_1(s)+L_A(s),u_2(s)+Z^{\varepsilon}_A(s))-g(v_1(s)+L_A(s),v_2(s)+Z^{\varepsilon}_A(s))\|ds\nonumber\\
\leq\!\!\!\!\!\!\!\!&&\frac{CT}{\varepsilon}\sup_{0\leq s\leq T}\left[\|u_1(s)-v_1(s)\|+\|u_2(s)-v_2(s)\|\right]\nonumber\\
\leq\!\!\!\!\!\!\!\!&&\frac{CT}{\varepsilon}d(u,v).\label{A4.8}
\end{eqnarray}
By \eref{A4.6} to \eref{A4.8} and choosing $T$ small enough, it is easy to see statement $(b)$ holds. Finally, system \eref{A4.1} has a unique solution in $S$ by the Banach fixed point theorem.

Let $(W_{\cdot}, V_{\cdot})\in S$ be the solution obtained by the above, for every $\sigma\in [\frac{1}{6}, \frac{1}{2}]$,
\begin{eqnarray*}
\|W^{\vare}_t\|_{2\sigma}\leq\!\!\!\!\!\!\!\!&& Ct^{-\sigma}\|x\|+C\int^t_0(t-s)^{-\sigma}\|N(W^{\vare}_s+L_A(s))\|ds+C\int^t_0 (t-s)^{-\sigma}ds\nonumber\\
\leq\!\!\!\!\!\!\!\!&&Ct^{-\sigma}\|x\|+C\int^t_0(t-s)^{-\sigma}\left[\|W^{\vare}_s\|^3_{1/3}+(K^{\vare}_1)^3+1\right]ds+Ct^{1-\sigma}\nonumber\\
\leq\!\!\!\!\!\!\!\!&&Ct^{-\sigma}\|x\|+C\int^t_0(t-s)^{-\sigma}\left[s^{-1/2}B^3+(K^{\vare}_1)^3+1\right]ds+Ct^{1-\sigma}\nonumber\\
\leq\!\!\!\!\!\!\!\!&&C(1+t^{-\sigma}),
\end{eqnarray*}
where $C$ is some constant depending on $\|x\|$, $\sigma$ and $K^{\vare}_1(\omega)$. The proof is complete.
\end{proof}

Now, we give a position to prove Theorem \ref{EUS}.

\textbf{The proof of Theorem \ref{EUS}:} By Lemma \ref{AL4.1}, Eq. \eref{A4.1} admits a unique local solution $W^{\vare}\in C([0,T]; H)\cap C((0,T]; V)$, $V^{\vare}\in C([0,T]; H)$ for some $T>0$, then follow the same steps in \cite[Lemma 4.2]{Xu1}, we can extend this solution to be $W^{\vare}\in C([0,\infty); H)\cap C((0,\infty); V)$, $V^{\vare}\in C([0,\infty); H)$.

We now prove that the uniqueness of the solution. Suppose there are two solutions $W^{\vare}\in C([0,T]; H)\cap C((0,T]; V)$, $V^{\vare}\in C([0,T]; H)$ and $\tilde W^{\vare}\in C([0,T]; H)\cap C((0,T]; V)$, $\tilde V^{\vare}\in C([0,T]; H)$. Thanks to the uniqueness of $[0, T]$, we have $W^{\vare}_T=\tilde W^{\vare}_T$ and $V^{\vare}_T=\tilde V^{\vare}_T$. For any $T_0>T$, it follows from the continuity that
$$
\sup_{T\leq t\leq T_0}\|W^{\vare}_t\|_1\leq \tilde C,\quad\quad \sup_{T\leq t\leq T_0}\|\tilde W^{\vare}_t\|_1\leq \tilde C,
$$
where $\tilde C>0$ depends on $T_0, \omega$. Hence, for all $t\in[T, T_0]$, we have
\begin{eqnarray}
&&\|W^{\vare}_t-\tilde W^{\vare}_t\|_{1}\nonumber\\
\leq\!\!\!\!\!\!\!\!&& C\int^t_T (t-s)^{-1/2}\|N(W^{\vare}_s+L_A(s))-N(\tilde W^{\vare}_s+L_A(s))\|ds\nonumber\\
\!\!\!\!\!\!\!\!&&+C\int^t_T(t-s)^{-1/2}\|f(W^{\vare}_s+L_A(s),V^{\vare}_s+Z^{\varepsilon}_A(s))-f(\tilde W^{\vare}_s+L_A(s),\tilde V^{\vare}_s+Z^{\varepsilon}_A(s))\|ds\nonumber\\
\leq\!\!\!\!\!\!\!\!&& C\int^t_T (t-s)^{-1/2}\left[1+(K^{\vare}_1)^2+\|W^{\vare}_s\|^2_1+\|\tilde W^{\vare}_s\|^2_1\right]\|W^{\vare}_s-\tilde W^{\vare}_s\|_{1}ds\nonumber\\
\!\!\!\!\!\!\!\!&&+C\int^t_T(t-s)^{-1/2}\left[\|W^{\vare}_s-\tilde W^{\vare}_s\|+\|V^{\vare}_s-\tilde V^{\vare}_s\|\right]ds\label{A4.10}\\
\leq\!\!\!\!\!\!\!\!&& C\tilde K^{\vare}_1\int^t_T (t-s)^{-1/2}\|W^{\vare}_s-\tilde W^{\vare}_s\|_{1}ds+C\int^t_T(t-s)^{-1/2}\left[\|W^{\vare}_s-\tilde W^{\vare}_s\|+\|V^{\vare}_s-\tilde V^{\vare}_s\|\right]ds,\nonumber
\end{eqnarray}
where $\tilde K^{\vare}_1=1+(K^{\vare}_1)^2+2\tilde C^2$. On the other hand,
\begin{eqnarray}
\|V^{\vare}_t-\tilde V^{\vare}_t\|\leq\!\!\!\!\!\!\!\!&& \frac{1}{\varepsilon}\int^t_T\|g(W^{\vare}_s+L_A(s),V^{\vare}_s)+Z^{\varepsilon}_A(s))-g(\tilde W^{\vare}_s+L_A(s),\tilde V^{\vare}_s+Z^{\varepsilon}_A(s))\|ds\nonumber\\
\leq\!\!\!\!\!\!\!\!&&\frac{C}{\varepsilon}\int^t_T\left[\|W^{\vare}_s-\tilde W^{\vare}_s\|+\|V^{\vare}_s-\tilde V^{\vare}_s\|\right]ds.\label{A4.11}
\end{eqnarray}
By \eref{A4.10} and \eref{A4.11}, Gronwall's inequality implies $W^{\vare}_t=\tilde W^{\vare}_t$ and $V^{\vare}_t=\tilde V^{\vare}_t$ for all $t\in[T_0, T]$. Since $T_0$ is arbitrary, we get the uniqueness of the solution.

For any $\varepsilon>0$ fixed, $L_A\in D([0, \infty); V)$  and $Z^{\vare}_A(t)\in H$. By the discussion above, $X^{\varepsilon}_t=W^{\vare}_t+L_A(t)$, $Y^{\varepsilon}_t=V^{\vare}_t+Z^{\vare}_A(t)$ is the unique solution to \eref{main equation}.$\Box$

\subsection{Galerkin approximation}

Recall the Galerkin approximation of system \eref{main equation}:
\begin{equation}\left\{\begin{array}{l}\label{A4.12}
\displaystyle
dX^{m,\vare}_t=[AX^{m,\vare}_t+N^m(X^{m,\vare}_t)+f^m(X^{m,\vare}_t, Y^{m,\vare}_t)]dt+d\bar L^m_t,\quad X^{m,\vare}_0=x^{m}\in H_m\\
dY^{m,\vare}_t=\frac{1}{\vare}[AY^{m,\vare}_t+g^m(X^{m,\vare}_t, Y^{m,\vare}_t)]dt+\frac{1}{\vare^{1/\alpha}}d\bar Z^m_t,\quad Y^{m,\vare}_0=y^m\in H_m.\end{array}\right.
\end{equation}

\begin{theorem} \label{GA}
For every $x\in H$, $y\in H$, system \eref{A4.12} has a unique mild solution $X^{m,\vare}.(\omega)\in D([0,\infty);H)\cap D((0,\infty);V)$ and $Y^{m,\vare}_t(\omega)\in H$, i.e.,
\begin{equation}\left\{\begin{array}{l}\label{A mild solution of F}
\displaystyle
X^{m,\varepsilon}_t=e^{tA}x^m+\int^t_0e^{(t-s)A}N^m(X^{m,\varepsilon}_s)ds+\int^t_0e^{(t-s)A}f^m(X^{m,\varepsilon}_s, Y^{m,\varepsilon}_s)ds+\int^t_0 e^{(t-s)A}d\bar {L}^m_s,\nonumber\\
Y^{m,\varepsilon}_t=e^{tA/\varepsilon}y^m+\frac{1}{\varepsilon}\int^t_0e^{(t-s)A/\varepsilon}g^m(X^{m,\varepsilon}_s,Y^{m,\varepsilon}_s)ds
+\frac{1}{\vare^{1/\alpha}}\int^t_0 e^{(t-s)A/\varepsilon}d\bar {Z}^m_s.\nonumber
\end{array}\right.
\end{equation}
Moreover, $\PP$-a.s.,
\begin{eqnarray}
\lim_{m\rightarrow \infty}\|X^{m,\vare}_t-X^{\vare}_t\|=0,\quad t\geq 0, \label{FA1}
\end{eqnarray}
\begin{eqnarray}
\lim_{m\rightarrow \infty}\|X^{m,\vare}_t-X^{\vare}_t\|_1=0,\quad t>0. \label{FA2}
\end{eqnarray}
\end{theorem}

\begin{proof}
Following the same argument in proving Theorem \ref{EUS}, it is easy to prove that the system \eref{A4.12} has a unique mild solution $X^{m,\vare}.(\omega)\in D([0,\infty);H)\cap D((0,\infty);V)$ and $Y^{m,\vare}_t\in H$. It remains to prove \eref{FA1} and \eref{FA2}. Since when $t=0$, \eref{FA1} holds obviously. We only show \eref{FA2}.

It is easy to see that for any $T>0$,
$$
\sup_{0\leq t\leq T}[\|X^{\vare}_t\|+\|Y^{\vare}_t\|]\leq \tilde B,
$$
where $\tilde B>0$ depends on $\|x\|$, $\|y\|$, $T$ and $K^{\vare}_t$. By property \eref{P2}, for any $\theta\in(1/2,1]$, $t\in (0, T]$
\begin{eqnarray*}
\|X^{\varepsilon}_t\|_{\theta}
\leq\!\!\!\!\!\!\!\!&&\|e^{tA}x\|_{\theta}+\left\|\int^t_0e^{(t-s)A}N(X^{\varepsilon}_s)ds\right\|_{\theta}+\left\|\int^t_0e^{(t-s)A}f(X^{\varepsilon}_s, Y^{\varepsilon}_s)ds\right\|_{\theta}+\left\|\int^t_0 e^{(t-s)A}dL_s\right\|_{\theta}\\
\leq \!\!\!\!\!\!\!\!&&t^{-\theta/2}\|x\|+\int^t_0(t-s)^{-1/2}\|N(X^{\varepsilon}_s)\|_{-(1-\theta)}ds
+
C\int^t_0 (t-s)^{-\theta/2}ds+\|L_{A}(t)\|_{\theta} \nonumber\\
\leq \!\!\!\!\!\!\!\!&& t^{-\theta/2}\|x\|+K^{\vare}_t+C_T+C\int^t_0(t-s)^{-1/2}(1+\|X^{\varepsilon}_s\|^3_{\theta/3})ds  \nonumber\\
\leq \!\!\!\!\!\!\!\!&&t^{-\theta/2}\|x\|+K^{\vare}_t+C_T+\int^t_0(t-s)^{-1/2}\|X^{\varepsilon}_s\|^2 \|X^{\varepsilon}_s\|_{\theta}ds \nonumber\\
\leq \!\!\!\!\!\!\!\!&&t^{-\theta/2}\|x\|+K^{\vare}_t+C_T+C\sup_{0\leq s\leq t}\|X^{\varepsilon}_t\|^2\int^t_0(t-s)^{-1/2} \|X^{\varepsilon}_s\|_{\theta}ds,\nonumber\\
\leq \!\!\!\!\!\!\!\!&&t^{-\theta/2}\|x\|+K^{\vare}_t+C_T+C\tilde{B}^2\int^t_0(t-s)^{-1/2} \|X^{\varepsilon}_s\|_{\theta}ds.\nonumber\\
\leq \!\!\!\!\!\!\!\!&&t^{-\theta/2}\|x\|+K^{\vare}_t+C_T+C\tilde{B}^2\left[\int^t_0(t-s)^{-p/2}ds\right]^{1/p} \left(\int^t_0\|X^{\varepsilon}_s\|^q_{\theta}ds\right)^{1/q},
\end{eqnarray*}
where $\frac{1}{p}+\frac{1}{q}=1$ with $1<p<2$. Then we have
\begin{eqnarray*}
\|X^{\varepsilon}_t\|^q_{\theta}
\leq\!\!\!\!\!\!\!\!&& C\left(t^{-\theta/2}\|x\|+K^{\vare}_t+C_T\right)^q+C_T\tilde{B}^{2q}\int^t_0\|X^{\varepsilon}_s\|^q_{\theta}ds.
\end{eqnarray*}
The Gronwall's inequality yields
\begin{eqnarray}
\|X^{\vare}_t\|_{\theta}\leq \tilde C (t^{-\theta/2}+1),\label{A4.15}
\end{eqnarray}
where $\tilde C>0$ depends on $\|x\|$, $\|y\|$, $T$ and $K^{\vare}_t$. Observe
\begin{eqnarray}
X^{m,\vare}_t-X^{\vare}_t=\!\!\!\!\!\!\!\!&&e^{tA}(x^m-x)+\int^{t}_{0}e^{(t-s)A}(\pi_m-I)N(X^{\vare}_s)ds\nonumber\\
&&\!\!\!\!\!\!\!\!+\int^{t}_{0}e^{(t-s)A}(N^{m}(X^{m,\vare}_s)-N^{m}(X^{\vare}_s))ds+\int^{t}_{0}e^{(t-s)A}(\pi_m-I)f(X^{\vare}_s, Y^{\vare}_s)ds\nonumber\\
&&\!\!\!\!\!\!\!\!+\int^{t}_{0}e^{(t-s)A}(f^{m}(X^{m,\vare}_s,Y^{m,\vare}_s)-f^{m}(X^{\vare}_s,Y^{\vare}_s))ds
+\left[\bar L^{m}_A(t)-L_A(t)\right]\nonumber\\
:=\!\!\!\!\!\!\!\!&&\sum^6_{i=1}I_i(t).\label{A4.16}
\end{eqnarray}
It is clear that
\begin{equation}
\lim_{m\to\infty}\|I_1(t)\|_{1}=0,\quad \lim_{m\to\infty}\|I_6(t)\|_{1}=0.\label{A4.17}
\end{equation}

For $I_4(t)$, notice that $f$ is bounded and by the dominated convergence theorem,
\begin{equation}
\|I_4(t)\|_{1}\leq C\int^t_0(t-s)^{-1/2}\|(\pi_m-I)f(X^{\vare}_s, Y^{\vare}_s)\|ds\rightarrow 0,\quad m\to\infty.\label{A4.18}
\end{equation}

For $I_5(t)$,
\begin{equation}
\|I_5(t)\|_{1}\leq C\int^t_0(t-s)^{-1/2}\left(\|X^{m,\vare}_s-X^{\vare}_s\|+\|Y^{m,\vare}_s-Y^{\vare}_s\|\right)ds.\label{A4.19}
\end{equation}

For $I_2(t)$, the property \eref{P2} and \eref{A4.15} imply for any $0<s<t$
$$\lim_{m\to\infty}\|(\pi_m-I)N(X^{\vare}_s)\|= 0.$$
This and the dominated convergence theorem yield
\begin{equation}
\|I_2(t)\|_{1}\leq C\int^t_0(t-s)^{-1/2}\|(\pi_m-I)N(X^{\vare}_s)\|ds\rightarrow 0,\quad m\to\infty.\label{A4.20}
\end{equation}

It remains to estimate $I_3(t)$. By \eref{P4}, \eref{A4.15} and the interpolation inequality,
\begin{eqnarray}
\|I_3(t)\|_{1}
\leq\!\!\!\!\!\!\!\!&&C\int^t_0(t-s)^{-1/2}\|N^{m}(X^{m,\vare}_s)-N^{m}(X^{\vare}_s)\|ds\nonumber\\
\leq\!\!\!\!\!\!\!\!&&C\int^t_0(t-s)^{-1/2}\|N(X^{m,\vare}_s)-N(X^{\vare}_s)\|ds\nonumber\\
\leq\!\!\!\!\!\!\!\!&&C\int^t_0(t-s)^{-1/2}(1+\|X^{m,\vare}_s\|^2_{1/3}+\|X^{\vare}_s\|^2_{1/3})\|X_s-X^{m,\vare}_s\|_{1/3}ds\nonumber\\
\leq\!\!\!\!\!\!\!\!&&C\int^t_0(t-s)^{-1/2}(1 +\|X^{\vare}_s\|^{4/3}\|X^{\vare}_s\|^{2/3}_{1}+\|X^{m,\vare}_s\|^{4/3}\|X^{m,\vare}_s\|^{2/3}_{1})\|X^{m,\vare}_s-X^{\vare}_s\|_{1}ds\nonumber\\
\leq\!\!\!\!\!\!\!\!&&C\int^t_0(t-s)^{-1/2}(1+2\tilde B^{4/3}\tilde C s^{-1/3})
\|X^{\vare}_s-X^{m,\vare}_s\|_{1}ds.\label{A4.21}
\end{eqnarray}

On the other hand,
\begin{eqnarray}
Y^{m,\vare}_t-Y^{\vare}_t=\!\!\!\!\!\!\!\!&&e^{tA/\vare}(y^m-y)+\frac{1}{\vare}\int^{t}_{0}e^{(t-s)A/\vare}(\pi_m-I)g(X^{\vare}_s, Y^{\vare}_s)ds \nonumber\\ &&\!\!\!\!\!\!\!\!+\frac{1}{\vare}\int^{t}_{0}e^{(t-s)A/\vare}(g^{m}(X^{m,\vare}_s,Y^{m,\vare}_s)-g^{m}(X^{\vare}_s,Y^{\vare}_s))ds+\left[\bar Z^{m,\vare}_A(t)-Z^{\vare}_A(t)\right]\nonumber\\
:=\!\!\!\!\!\!\!\!&&\sum^4_{i=1}J_i(t),\label{A4.22}
\end{eqnarray}
where $Z^{m,\vare}_A(t):=\frac{1}{\vare^{1/\alpha}}\int^t_0 e^{(t-s)A/\vare}d \bar{Z}^m_s$. It is clear that
\begin{equation}
\lim_{m\to\infty}\|J_1(t)\|=0,\quad \lim_{m\to\infty}\|J_4(t)\|=0.\label{A4.23}
\end{equation}

For $J_2(t)$, by the dominated convergence theorem, we obtain
\begin{equation}
\|J_2(t)\|\leq C\int^t_0\|(\pi_m-I)g(X^{\vare}_s, Y^{\vare}_s)\|ds\rightarrow 0,\quad m\to\infty.\label{A4.24}
\end{equation}

For $J_3(t)$, by the Lipschitz continuous of $g$,
\begin{equation}
\|J_3(t)\|\leq C_\epsilon\int^t_0\left(\|X^{m,\vare}_s-X^{\vare}_s\|+\|Y^{m,\vare}_s-Y^{\vare}_s\|\right)ds.\label{A4.25}
\end{equation}
By \eref{A4.22}-\eref{A4.25}, Fatou's lemma and Gronwall's inequality, we have
\begin{eqnarray}
\limsup_{m\to\infty}\|Y^{m,\vare}_t-Y^{\vare}_t\|\leq C_\epsilon\int^t_0\limsup_{m\to\infty}\|X^{m,\vare}_s-X^{\vare}_s\|ds.\label{A4.26}
\end{eqnarray}

Combing \eref{A4.16}-\eref{A4.21} and \eref{A4.26}, then by Fatou's lemma, we obtain
\begin{eqnarray*}
\limsup_{m\to\infty}\|X^{m,\vare}_t-X^{\vare}_t\|_{1}
\leq\!\!\!\!\!\!\!\!&&C\int^t_0(t-s)^{-1/2}\left(\limsup_{m\to\infty}\|X^{m,\vare}_s-X^{\vare}_s\|+\int^s_0\limsup_{m\to\infty}\|X^{m,\vare}_r-X^{\vare}_r\|dr\right)ds\\
\!\!\!\!\!\!\!\!&&+C\int^t_0(t-s)^{-1/2}(1+2\tilde B^{4/3}\tilde C s^{-1/3})\limsup_{m\to\infty}\|X^{\vare}_s-X^{m,\vare}_s\|_{1}ds\\
\leq\!\!\!\!\!\!\!\!&&C\int^t_0(t-s)^{-1/2}(1+2\tilde B^{4/3}\tilde C s^{-1/3})\limsup_{m\to\infty}\|X^{\vare}_s-X^{m,\vare}_s\|_1ds\\
\leq\!\!\!\!\!\!\!\!&&C\left[\int^t_0(t-s)^{-p/2}(1+2\tilde B^{4/3}\tilde C s^{-1/3})^pds\right]^{1/p}\left[\int^t_0\limsup_{m\to\infty}\|X^{\vare}_s-X^{m,\vare}_s\|^q_1ds\right]^{1/q}.
\end{eqnarray*}
where $\frac{1}{p}+\frac{1}{q}=1$ with $1<p<\frac{6}{5}$. Then we have
\begin{eqnarray*}
\limsup_{m\to\infty}\|X^{m,\vare}_t-X^{\vare}_t\|^q_{1}\leq\!\!\!\!\!\!\!\!&&\hat C\int^t_0\limsup_{m\to\infty}\|X^{\vare}_s-X^{m,\vare}_s\|^q_1ds,
\end{eqnarray*}
where $\hat C$ depends on $\|x\|$, $\|y\|$, $T$ and $K^{\vare}_t$. Hence, the Grownwall's inequality implies
$$\limsup_{m\rightarrow\infty}\|X^{m,\vare}_t-X^{\vare}_t\|_{1}=0.$$
The proof is complete.
\end{proof}

\textbf{Acknowledgment}. We would like to gratefully thank Professor Lihu Xu for valuable discussions and kind suggestions. This work was conducted during the first author visited
the Department of Mathematics, Faculty of Science and Technology, University
of Macau, he thanks for the finance support and hospitality. Xiaobin Sun is supported by the National Natural Science Foundation of China (11601196, 11771187, 11931004), the NSF of Jiangsu Province (No. BK20160004) and the Priority Academic Program Development of Jiangsu Higher Education Institutions. Jianliang Zhai is supported by the National Natural Science Foundation of China (11431014, 11671372, 11721101), the Fundamental Research Funds for the Central Universities (No. WK0010450002, WK3470000008), Key Research Program of Frontier Sciences, CAS, No: QYZDB-SSW-SYS009, School Start-up Fund (USTC) KY0010000036.

\end{document}